\documentclass[a4paper]{easychair}

\usepackage{microtype}

\usepackage{amsmath}
\usepackage{amssymb}
\usepackage{proof}
\usepackage[all]{xy}
\usepackage{graphicx}

\newcommand{\calC}{\mathcal{C}}
\newcommand{\calE}{\mathcal{E}}

\newcommand{\calP}{\mathcal{P}}
\newcommand{\calQ}{\mathcal{Q}}

\newcommand{\sop}{[}
\newcommand{\scl}{]}
\newcommand{\sel}[2]{#1 / #2}
\newcommand{\unsubst}[2]{\sop \sel{#2}{#1} \scl}

\usepackage{mathabx}
\newcommand{\impl}{\rightarrow}
\newcommand{\dual}[1]{\overline{#1}}
\newcommand{\Lor}{\bigvee}

\newcommand{\Shallow}{\mathsf{Sh}}
\newcommand{\Deep}{{\mathsf{Dp}}}
\newcommand{\Br}{\mathsf{Br}}

\newcommand{\Exp}{\mathrm{Exp}}

\newcommand{\omerge}{\sqcup}

\newcommand{\mergered}{\stackrel{\omerge}{\rightarrow}}
\newcommand{\LK}{\ensuremath{\mathbf{LK}}}
\newcommand{\LKE}{\ensuremath{\mathbf{LKE}}}

\newcommand{\lkforall}[3][]{\infer[\forall{#1}]{#2}{#3}}
\newcommand{\lkexists}[3][]{\infer[\exists{#1}]{#2}{#3}}
\newcommand{\lkand}[3][]{\infer[\land{#1}]{#2}{#3}}
\newcommand{\lkor}[3][]{\infer[\lor{#1}]{#2}{#3}}

\newcommand{\lkcut}[3][]{\infer[\mathrm{cut}{#1}]{#2}{#3}}

\newcommand{\rk}{\ensuremath{\mathrm{rk}}}   % rank
\newcounter{tmprecount}
\title{Expansion Trees with Cut}
\author{Federico Aschieri\thanks{Funded by FWF Lise Meitner Grant M 1930--N35 and START project Y544--N23} \and Stefan Hetzl\thanks{Funded by the WWTF Vienna Research Group (VRG) 12-004} \and Daniel Weller}
\institute{Institute of Discrete Mathematics and Geometry\\
Vienna University of Technology\\
Vienna, Austria\\
\texttt{\{federico.aschieri,stefan.hetzl,daniel.weller\}@tuwien.ac.at}}
\authorrunning{F. Aschieri S.~Hetzl and D.~Weller}
\titlerunning{Expansion Trees With Cut} %mandatory. First: Use abbreviated first/middle names. Second (only in severe cases): Use first author plus 'et. al.'

\newtheorem{theorem}{Theorem}
\newtheorem{lemma}[theorem]{Lemma}
\newtheorem{definition}{Definition}
\newtheorem{example}{Example}
\newtheorem{remark}{Remark}

\begin{document}
\maketitle

\begin{abstract}
Herbrand's theorem is one of the most fundamental insights in logic. From the syntactic point of view, it suggests a compact representation of proofs in classical first- and higher-order logic by recording the information of which instances have been chosen for which quantifiers. \\
This compact representation is known in the literature as Miller's expansion tree proof. It is inherently analytic and hence corresponds to
a cut-free sequent calculus proof. Recently several extensions of
such proof representations to proofs with cuts have been proposed. These
extensions are based on graphical formalisms similar to proof nets
and are limited to prenex formulas.\\
In this paper we present a new syntactic approach that directly extends Miller's expansion trees by cuts and covers also non-prenex formulas. We describe a cut-elimination procedure for our expansion trees with cut that is based on the natural reduction steps and show that it is weakly normalizing.
\end{abstract}

\section{Introduction}

Herbrand's theorem~\cite{Herbrand30Recherches,Buss95Herbrand}, one of the
most fundamental insights of logic, characterizes the
validity of a formula in classical first-order logic by the existence of a propositional
tautology composed of instances of that formula.

From the syntactic point of view this theorem induces a way of describing proofs:
by recording which instances have been picked for which quantifiers we
obtain a description of a proof up to its propositional part, a part we
often want to abstract from. An example for a formalism that carries out
this abstraction are Herbrand proofs~\cite{Buss95Herbrand}.
This generalizes nicely to most classical systems with quantifiers,
in particular to simple type theory, as in Miller's \emph{expansion tree proofs}~\cite{Miller87Compact}.
Such formalisms are compact and useful proof certificates in many situations; they
are for example produced naturally by methods of automated deduction such
as instantiation-based reasoning~\cite{Korovin09Instantiation} and they play
a central role in many proof transformations in the GAPT-system~\cite{Ebner16System}.

These formalisms consider only instances of the formula that has been proved
and hence are {\em analytic} proof formalisms, corresponding to cut-free proofs
in the sequent calculus. Considering an expansion tree to be a compact
representation of a proof, it is thus natural to ask about the possibility of extending this kind of representation
to {\em non-analytic} proofs, corresponding to proofs with cut in the sequent
calculus.

In addition to enlarging the scope of instance-based proof representations, the
addition of cuts to expansion trees promises to shed more light on the computational
content of classical logic. This is a central topic of proof theory and has
therefore attracted considerable attention, see ~\cite{Parigot92LambdaMu, AschieriZ16, Danos97New,Curien00Duality},
\cite{Urban00Classical,Urban01Strong},
 or~\cite{Baaz00CutElimination}, for
different investigations in this direction.% and~\cite{Avigad10Computational} for a survey covering classical arithmetic.

Two  proof formalisms manipulating only formula instances and incorporating a notion of cut have recently been proposed: proof
forests~\cite{Heijltjes10Classical} and Herbrand nets~\cite{McKinley13Proof}.
While some definitions in the setting of proof forests are motivated by the game semantics for classical arithmetic~\cite{Coquand95Semantics}, Herbrand nets are based on methods for
proof nets~\cite{Girard87Linear}.
%\footnote{TODO: Federico, could you say more about the motivations of these formalisms and how it compares to ours, in particular w.r.t.\ game semantics?}
These two formalisms share a number of
properties: both of them work in a graphical notation for proofs,
both deal with prenex formulas only, for both weak but no strong normalization results are known.

In this paper we present a  purely \emph{syntactic} approach to the topic.  We start from expansion tree proofs,  add cuts and define cut-reduction rules,  naturally extending the existing
literature in this tradition. The result is a \emph{rewriting theory of expansion trees  with cuts}. The main staple of a good rewriting theory is that the syntax should look simple and  the reduction rules should be as few and as elementary as possible. %think, for instance, about the $\lambda$-calculus. 
When a rewriting system falls short of any of these requirements, reasoning about its combinatorial properties may easily become unwieldy; when it satisfies them, it is always a good sign. Indeed, expansion trees are by design compact strings of symbols, expansion proofs just lists of those trees and  the reduction rules that we shall present straightforward manipulation of those lists. This is a novel technical achievement. In fact,  graph-based formalisms like proof forests and Herbrand nets allow rather simple mathematical definitions of tree forests and their transformations, but as soon as one tries to write them down syntactically, their rewriting complexity becomes evident. A  simple rewriting theory may help to solve the intricate combinatorial problems that arise, like strong normalization. 

With respect to proof forests, the main related work, we offer several technical novelties.\\

\emph{Miller's correctness criterion.} Expansion trees are just simple collections of witnesses for quantifiers, so not every tree makes logical and semantical sense. Miller's correctness criterion \cite{Miller87Compact} is the most direct known way to express that an expansion tree (list) is sound: it requires a certain acyclic ordering of the tree nodes, it maps the tree into a propositional formula and asks it to be a tautology. Syntactically, the definition of Miller's criterion follows in a straightforward way the tree's shape and the obtained propositional formula matches exactly the tree's number of leaves. Semantically, Miller's criterion  states that a list of expansion trees represents a winning strategy in Coquand's backtracking games \cite{Coquand95Semantics}. Though Heijltjes' correctness criterion was motivated as well by Coquand's game semantics, it represents a different way of extracting a propositional formula from a list of expansion trees: it is constructed from a case distinction on all cuts, rendering the formula's size  exponential with the respect to the number of cuts. On our side, we managed to keep Miller's correctness criterion unchanged. As by-product, we also effortlessly obtain a treatment of non-prenex formulas. This avoids not only the distortion of the intuitive meaning of a formula
by prenexification, but also the non-elementary increase in complexity that
can be caused by prenexification~\cite{Baaz99CutNormal}.
It also seems possible to extend proof forests and Herbrand
nets to non-prenex formulas, but this has not been done in~\cite{Heijltjes10Classical}.% and would still be subject to the limits discusses above.
\\

\emph{Cut-Reductions}. Eliminating cuts from expansion proofs resembles a Coquand game between expansion trees, when they are interpreted as strategies.  Following this game semantics analogy, one would thus  expect, during cut-elimination, to only encounter new trees whose branchings  are isomorphic to sub-trees of the original expansion proof. This however does not happen in the theory of proof forests and Herbrand nets: the restructuring performed during cut-elimination is significant and trees eventually become much bigger than the original ones, due to an operation of copying and glueing  them together. Though we are not pursuing the game semantics analogy here, we  define cut-reduction steps that instead satisfy the mentioned condition. The gain is all about the rewriting theory of expansion proofs: cut-reduction only involves  the operations of copying, decomposing, substituting terms and renaming variables applied to subtrees of the original ones. \\

\emph{Pruning and Bridges}. In proof forests an unexpected technical issue arises. Cut-reductions create some unwanted ``bridges'' that cause non-termination of cut-elimination. Therefore, additional restructuring of the forest is needed, this time in terms of scissors, cutting those bridges. Here we show that bridges are not an issue at all and our cut-elimination terminates, regardless of bridges. Again, in this way we avoid an additional layer of rewriting complexity. \\

\subsection{Plan of the Paper}
In Section \ref{sec:expansion}, we modify Miller's concept of expansion proof in order to also include special pairs of expansion trees, which  represent logical cuts. In Section \ref{sec:completeness}, we show that our expansion proofs are sound and complete with respect to first-order classical validity. In Section  \ref{sec:cutelim}, we define a cut-elimination procedure which transform any expansion proof with cuts into a cut-free one.
%
% 2nd paragraph: graph notation vs. tree notation

\section{Expansion Trees}\label{sec:expansion}

In this entire paper we work with classical first-order logic. Formulas and
terms are defined as usual. In order to simplify the exposition, we restrict
our attention to formulas in negation normal form (NNF) and without vacuous quantifiers. Mutatis mutandis all
notions and results of this paper generalize to arbitrary formulas. We write
$\dual{A}$ for the de Morgan dual of a formula $A$. A {\em literal} is an
atom $P(t_1,\ldots,t_n)$ or a negated atom $\dual{P}(t_1,\ldots,t_n)$.
We start by defining Miller's concept of expansion tree \cite{Miller87Compact}. 
\begin{definition}[Expansion Trees]\label{def:exptree}
Expansion trees and a function $\Shallow(\cdot)$ (for {\em shallow}) that maps
an expansion tree to a propositional formula are defined inductively as follows:
\begin{enumerate}
\item A literal $L$ is an expansion tree with $\Shallow(L)=L$.
\item If $E_1$ and $E_2$ are expansion trees and $\circ \in \{ \land,\lor \}$, then
$E_1\circ E_2$ is an expansion tree with $\Shallow(E_1\circ E_2) = \Shallow(E_1)\circ \Shallow(E_2)$.
\item If $ t_1,\ldots,t_n $ is a sequence of terms and $E_1,\ldots,E_n$ are
expansion trees with $\Shallow(E_i) = A\unsubst{x}{t_i}$
for $i=1,\ldots,n$, then
$E = \exists x\, A +^{t_1} E_1 \cdots +^{t_n} E_n$ is an expansion tree with
$\Shallow(E) = \exists x\, A$.
\item If $E_0$ is an expansion tree with $\Shallow(E_0)=A\unsubst{x}{\alpha}$, then
$E = \forall x\, A +^\alpha E_0$ is an expansion tree with $\Shallow(E) = \forall x\, A$.
\end{enumerate}
The $+^{t_i}$ of point 3. are called {\em $\exists$-expansions}
and the $+^\alpha$ or point 4. are called {\em $\forall$-expansions},
and both $\forall$- and $\exists$-expansions
are called {\em expansions}.
The variable $\alpha$ of a $\forall$-expansion $+^\alpha$ is called
{\em eigenvariable} of this expansion. 
%and we define $\Shallow(+^{t_i})=\Shallow(E)$.
We say that $+^{t_i}$ {\em dominates} all the expansions
in $E_i$. Similarly, $+^\alpha$ {\em dominates} all the expansions
in $E_0$.  We also say that $E$ is an {\em expansion tree of $\Shallow(E)$}. If $\calE=E_{1}, \dots, E_{n}$ is a sequence of expansion trees, we define $\Shallow(\calE)=\Shallow(E_{1}), \ldots, \Shallow(E_{n})$.

\end{definition}
We recall now the definition of the propositional formula $\Deep(E)$, which is used to state Miller's correctness criterion for an expansion tree $E$. \begin{definition}
We define the function $\Deep(\cdot)$ (for {\em deep}) that maps an expansion
tree to a propositional formula as follows:
\begin{align*}
\Deep(L) &= L\ \mbox{for a literal}\ L,\\
\Deep(E_1\circ E_2) &= \Deep(E_1)\circ \Deep(E_2)\ \mbox{for $\circ\in\{ \land, \lor \}$},\\
\Deep(\exists x\, A +^{t_1} E_1 \cdots +^{t_n} E_n) &= \Lor_{i=1}^n \Deep(E_i),\ \mbox{and}\\
\Deep(\forall x\, A +^{\alpha} E_0) &= \Deep(E_0).
\end{align*}
If $\calE=E_{1}, \dots, E_{n}$ is a sequence of expansion trees, we define $\Deep(\calE)=\Deep(E_{1}), \ldots, \Deep(E_{n})$.
\end{definition}
Cuts are simply defined as pairs of expansion trees, whose shallow formulas are one the involutive negation of the other.
\begin{definition}[Cut]
A {\em cut} is a set $C = \{E_1, E_2\}$ of two expansion trees s.t.\ 
$\Shallow(E_1) = \dual{\Shallow(E_2)}$.  A formula is called positive if its top connective is $\lor$ or $\exists$ or
a positive literal. An expansion tree $E$ is called positive if $\Shallow(E)$ is
positive. It will sometimes be useful to consider a cut as an ordered pair: to that aim
we will write a cut as $C = (E_1,E_2)$ with parentheses instead of curly braces
with the convention that $E_1$ is the positive expansion tree. For a cut
$C = (E_1,E_2)$, we define $\Shallow(C) = \Shallow(E_1)$ which is also called
{\em cut-formula} of $C$. We define $\Deep(C) = \Deep(E_1) \land \Deep(E_2)$. If $\calC=C_{1}, \dots, C_{n}$ is a sequence of cuts, we define $\Deep(\calC)=\Deep(C_{1}), \ldots, \Deep(C_{n})$ and $\Shallow(\calC)=\Shallow(C_{1}), \ldots, \Shallow(C_{n})$.
\end{definition}
For each expansion tree  we now define the set of finite formulas and number sequences, representing all formulas that one encounters  and all branch choices one makes in any complete path from the tree's root to one of its leaves. This concept will  soon be needed  for defining correctness of expansion proofs.

\begin{definition}[Formula Branch]\label{defi-fb}
We define a function $\Br(\cdot)$ (for {\em branch}) that maps an expansion tree with merges to a finite set $\{l_1,\ldots, l_k\}$, where each $l_i$ is some list made of formulas and the integers $1$ or $2$.
\begin{align*}
\Br(L) &= L\ \mbox{for a literal}\ L,\\
\Br(E_1\circ E_2) &= \{ \Shallow(E_{1}\circ E_{2}), i, s\ |\  s\in\Br(E_i) \}\ \mbox{for $\circ\in\{ \land, \lor \}$},\\
%\Br(E_1\circ E_2) &=\{ \Shallow(E_{1}\circ E_{2}), s\ |\  s\in\Br(E_1) \mbox{ or } s\in \Br(E_2)\}\ \mbox{for $\circ\in\{ \land, \lor \}$},\\
\Br(\exists x\, A +^{t_1} E_1 \cdots +^{t_n} E_n) &= \{\exists x\, A, s\ |\ s\in \Br(E_i)\mbox{ with $i\in\{1, \ldots, n\}$}\},\\
\Br(\forall x\, A +^{\alpha} E_0) &= \{\forall x\, A, s\ |\ s\in \Br(E_0)\},\\
%\Br(E_{1}\omerge E_{2}) &= \Br(E_{1})\cup \Br(E_{2}).
\end{align*}
For every cut $C=(E_{1}, E_{2})$ we define $\Br(C)=\Br(E_{1})\cup \Br(E_{2})$.
 
\end{definition}
A very simple property that we shall use without further mentioning is the following. 
\begin{lemma}\label{lemma-ins}
Let $E$  be an expansion tree and $F$ a sub-tree of $E$. Then there is a formula sequence $s$ such that for every $r\in \Br(F)$, it holds that $s, r\in \Br(E)$. 
\end{lemma}
\begin{proof}
By a straightforward induction on $E$.
\end{proof}

Expansion proofs will be defined as sequences of expansion trees and cuts satisfying a number of properties. The correctness criterion of expansion tree proofs~\cite{Miller87Compact},  as well as those of proof forests~\cite{Heijltjes10Classical} and Herbrand nets~\cite{McKinley13Proof}, have
two main components: 1.~a tautology-condition on one or more quantifier-free
formulas and 2.~an acyclicity condition on one or more orderings. These conditions can be interpreted, logically, as ensuring that expansion proofs represents logical proofs, semantically, as defining correct winning strategies in Coquand games with backtracking ~\cite{Heijltjes10Classical, AschieriGames}. While the tautology
condition of~\cite{Miller87Compact} generalizes to the setting of cuts in a straightforward
way, the acyclicity condition needs a bit more work: in the setting of cut-free
expansion trees it is enough to require the acyclicity of an order on the $\exists$-expansions.
%When reading back an expansion proof into a proof in the sequent calculus,
%each $\exists$-expansion corresponds to an existential quantifier inference (along
%a branch, but not globally, see~\cite{Baaz12Complexity}). 
%If $+^t < +^s$ then
%the existential inference corresponding to $t$ must be below the existential
%inference corresponding to $s$ in all sequentializations of the expansion proof.
%In this sense, the order relation on expansion terms give the {\em necessary}
%constraints a sequentialization has to fulfill. Consequently its acyclicity
%ensures the finite depth of the proof.
%
%In the setting of cut-free expansion proofs it is enough to speak about the order of
%$\exists$-expansions.
In our setting that includes cuts we also have to speak about the order of
cuts (w.r.t.\ each other and w.r.t.\ $\exists$-expansions).
%The cut being
%a binary rule, the relation of ``being above'' is sometimes
%not informative enough\meta{Stefan}{TODO: wann? Bsp geben}, instead we will
%speak about ``being above on the positive side'' and ``being above
%on the negative side'' where the positive / negative side of a
%cut is the side that contains the positive / negative copy of the cut formula.
To simplify our treatment of this order we also include $\forall$-expansions.
Together this leads to the following
%necessary 
inference ordering constraints.
\begin{definition}[Dependency Relation]\label{defi-dependency}
%Let $\calP = \calC, \calE$ be an expansion pre-proof.
Let  $\calP = \calC, \calE$, where  $\calC$ is a sequence of cuts and $\calE$ a sequence of expansion trees.
 We will define the
{\em dependency relation} $<_\calP$, which is a binary relation on the set
of expansions and cut \emph{occurrences} in $\calP$. First,
we define the binary relation $<^0_\calP$ (writing $<^0$ if $\calP$ is clear
from the context) as the least relation satisfying:
%
%This is a quantifier free
%theory working on statements of the form $x < y$, $c <^+ x$, and $c<^- x$
%where $x,y$ are $\exists$-expansions (i.e. plus-nodes) or $\forall$-expansions
%or cuts of $\calP$ and $c$ are cuts of $\calP$. The axioms of $\calO_\calP$
%are (with $c=(E^+,E^-)$ being a cut):
%
%
\begin{enumerate}
\item $v <^0 w$ if $w$ is an $\exists$-expansion in $\calP$ whose term contains the eigenvariable of the $\forall$-expansion $v$
\item $v <^0 w$ if $v$ is an expansion in $\calP$ that dominates the expansion $w$
\item $C <^0 v$ if $v$ is an expansion of the cut $C$ in $\mathcal{C}$
\item $v <^0 C$ if $C$ is  a cut  and $\Shallow(C)$ contains the eigenvariable of the $\forall$-expansion $v$
\end{enumerate}
$<_\calP$ is then defined to be the transitive closure of $<^0$. Again,
we write $<$ for $<_\calP$ if $\calP$ is 
clear from the context.
\end{definition}
%
%If $x,y$ are nodes or cuts in an expansion tree $\calP$, we will often simply
%write ``$x < y$'' instead of ``there is a derivation of $x < y$ from the
%axioms induced by $\calP$''.
%
\noindent As clauses 1--4 never relate two cuts, there is no $<_{\calP}$-cycle containing cuts only, thus $<_\calP$ is cyclic iff $w<_\calP w$ for an
expansion $w$: we will make use of this property without further mention. 

We now define the concept of expansion proof.
  In the following, lists of expansion trees and cuts will be identified modulo permutation of their elements.
\begin{definition}[Expansion Proofs]\label{def:exp_preproof}
Let $\calC=C_{1},\ldots, C_{n}$ be a sequence of cuts %with pairwise different cut-formulas
 and let
$\calE=E_{1},\ldots, E_{m}$ be a sequence of expansion trees. Let $\calP = \calC, \calE$. We define $\Shallow(\calP) = \Shallow(\calE)$, which
corresponds to the end-sequent of a sequent calculus proof, and
$\Deep(\calP) = \Deep(\calE),\Deep(\calC)$, which is a sequent of quantifier-free
formulas, and $\Br(\calP)=\Br(C_{1})\cup\ldots \cup\Br(C_{n})\cup \Br(E_{1})\cup\ldots \cup \Br(E_{m})$.%of pairwise different formulas. 
Then
$\calP$ is called {\em expansion proof} whenever:

\begin{enumerate}
\item \emph{(weak regularity)} for every $S$ and $T$ in $\calP$, if  $s, \forall x A, A\unsubst{x}{\alpha}, s'\in \Br(T)$ and $r, \forall x B, B\unsubst{x}{\alpha}, r'\in \Br(S)$, then  $s=r$,  $A=B$ and $S, T$ are both trees or both cuts.
\item (acyclicty) $<_\calP$ is acyclic, that is, $x<_\calP x$ holds for no $x$.
\item (validity) $\Deep(\calP)$ is a tautology, that is, a valid sequent.
\item (eigenvariable condition) For every $\forall$-expansion $+^{\alpha}$ in $\calP$, the variable $\alpha$ does not occur in $\Shallow(\calE)$.

\end{enumerate}
\end{definition}
%
%We are now ready to define the concept of expansion proof: it extends Miller's notion of  ``compact''  representation of cut-free proofs \cite{Miller87Compact}, in order to code proofs with cuts as well. 
%\begin{definition}\label{def:exp_proof} An {\em expansion proof} is an expansion pre-proof $\calP$ that satisfies the following condition: \begin{enumerate} \item  $\Shallow(\calP)$ does not contain free variables. \end{enumerate} \end{definition}
An important difference of expansion proofs with respect to other formalisms, such as proof
forests~\cite{Heijltjes10Classical} and Herbrand nets~\cite{McKinley13Proof}, is that the same $\forall$-expansion can occur multiple times. This phenomenon is very natural, as soon as one realizes that the  \emph{weak regularity} condition that we have imposed corresponds to an interpretation of eigenvariables as \emph{Skolem functions}.
 Namely, weak regularity  ensures: that the same witness is only used for the same formula with same parameters;  that an expansion proof can always be transformed into one satisfying the usual regularity condition that every $\forall$-expansion occurs exactly one time \cite{Miller87Compact}. This last property, that we shall not prove, %is a consequence of the soundness and completeness theorems of Section \ref{sec:completeness} and 
guarantees that we are still working with the familiar objects. Our weak regularity condition offers, however, a great technical advantage.  Namely, the definition of cut-reduction becomes much easier, as it avoids the heavy restructuring of the expansion trees which would be needed to prevent duplication of $\forall$-expansions. 
  
\noindent 
Condition 2 and 3 embody  Miller's correctness criterion.
Condition 4 could be formulated as asking that $\Shallow(\calP)$ does not contain free variables. But the real trouble is indeed that
 if $\calP$ is such that $\Shallow(\calP)$  contains free variables, then the eigenvariable of some $\forall$-expansion may be contained in $\Shallow(\calP)$, so that $\calP$ would not represent a proof of $\Shallow(\calP)$. This issue will become transparent in  Section \ref{sec:completeness}, where we show that the notion of expansion proof represents indeed a sound and complete proof system with respect to classical first-order validity. Moreover, again because that the eigenvariable of some $\forall$-expansion could be contained in $\Shallow(\calP)$, without condition 4 the  notion of expansion proof would  not  be closed under the cut-reduction that we shall provide in Section \ref{sec:cutelim}.
%
%In particular, an expansion proof being v-consistent does not contain a
%bridge in the sense of~\cite{Heijltjes10Classical}, i.e.\ a dependency from
%one side of a cut to its other side.
%
\begin{example}\label{ex.exp_preproof}
Consider the straightforward proof of $P(a)\impl \exists z\, Q(z)$ from
$\exists y\forall x\, (P(x) \impl Q(f(y)))$ via a cut on
$\forall x\exists y\, (P(x) \impl Q(f(y)))$. In negation normal form
these formulas are $\dual{P}(a)\lor \exists z\, Q(z)$,
$\exists y\forall x\, (\dual{P}(x) \lor Q(f(y)))$,
and $\forall x\exists y\, (\dual{P}(x)\lor Q(f(y)))$. The proof will
be represented by the expansion proof $\calP = \{ E^+, E^- \}, E_1, E_2$ where
\begin{align*}
E^+ = &\ \exists x\forall y\, ( P(x)\land \dual{Q}(f(y))) +^a (\ \forall y\, ( P(a) \land \dual{Q}(f(y)))
       +^\gamma P(a)\land\dual{Q}(f(\gamma))
        \ ) \\
E^- = &\ \forall x\exists y\, ( \dual{P}(x)\lor Q(f(y)) ) +^\beta (\ \exists y\, ( \dual{P}(\beta) \lor Q(f(y))) +^\alpha ( \dual{P}(\beta) \lor Q(f(\alpha)))
\ )\\
E_1 = &\ \forall y\exists x\, (P(x) \land \dual{Q}(f(y))) +^\alpha (\ \exists x\, (P(x)\land \dual{Q}(f(\alpha))) +^\beta P(\beta)\land\dual{Q}(f(\alpha)) 
\ )\\
E_2 = &\ \dual{P}(a) \lor ( \exists z\, Q(z) +^{f(\gamma)} Q(f(\gamma)) )
\end{align*}
We have $\Shallow(\calP) = \Shallow(E_1,E_2) =
\forall y\exists x\, (P(x) \land \dual{Q}(f(y))), \dual{P}(a) \lor \exists z\, Q(z)$
and
\begin{align*}
\Deep(\calP) =&\ \Deep(E^+)\land \Deep(E^-), \Deep(E_1), \Deep(E_2)\\
 =&\ (P(a) \land \dual{Q}(f(\gamma))) \land ( \dual{P}(\beta) \lor Q(f(\alpha))), P(\beta)\land\dual{Q}(f(\alpha)), \dual{P}(a)\lor Q(f(\gamma))
\end{align*}
 Note
that $\Deep(\calP) $
is a tautology (of the form $A\land B, \dual{B}, \dual{A}$). Let us now consider
the dependency relation induced by $\calP$: in $\calP$ each term belongs to at most one $\exists$-
and at most one $\forall$-expansion.
In such a situation we can uniformly write all expansions as $Qt$ for some
term $t$ and $Q\in\{ \exists, \forall \}$. The expansions of $\calP$ are
then written as $\exists a$, $\forall \gamma$, $\forall \beta$, $\exists \alpha$,
$\forall \alpha$, $\exists \beta$, and $\exists f(\gamma)$. Furthermore,
$\calP$ contains a single cut $C$. Then $<^0$ is exactly:
\begin{enumerate}
 \item $\forall \gamma <^0 \exists f(\gamma)$, $\forall \beta <^0 \exists \beta$,
$\forall \alpha <^0 \exists \alpha$,
\item $\exists a <^0 \forall \gamma$, $\forall \beta <^0 \exists \alpha$, $\forall \alpha <^0 \exists \beta$,
\item $C <^0 \exists a$, $C <^0 \forall \gamma$, $C <^0 \forall \beta$, $C <^0 \exists \alpha$,
\item there is no $v<^0 C$ as the cut formula of $C$ is variable-free.
\end{enumerate}
As the reader is invited to verify, $<$ is acyclic.
\end{example}

\section{Expansion Proofs and Sequent Calculus}\label{sec:completeness}

In this section we will clarify the relationship between our expansion proofs
and the sequent calculus. The concrete version of sequent calculus is of
no significance to the results presented here, they
hold mutatis mutandis for every version that is common in the literature.
For technical convenience, we treat sequents as multisets of formulas in Section \S \ref{complete} and as sets of formulas in Section \S \ref{sound}. 
%we choose a calculus where a sequent is a set of formulas and all rules are invertible.
%
\begin{definition}
The calculus {\LK} is defined as follows: initial sequents are of the form
$\Gamma, A, \dual{A}$ for an atom $A$. The inference rules are
\[
\lkforall{\Gamma, \forall x\, A}{\Gamma, A\unsubst{x}{\alpha}}
\quad 
\lkexists{\Gamma, \exists x\, A}{\Gamma, \exists x\, A, A\unsubst{x}{t}}
\quad
\lkand{\Gamma, A\land B}{\Gamma, A & \Gamma, B}
\quad
\lkor{\Gamma, A\lor B}{\Gamma, A, B}
\quad
\lkcut{\Gamma}{\Gamma, A & \dual{A}, \Gamma}
\]
with the usual side conditions: $\alpha$ must not appear in $\Gamma, \forall x\, A$
and $t$ must not contain a variable which is bound in $A$.
\end{definition}
%
%Due to the global nature of expansion proofs, they correspond to regular {\LK}-proofs.
An {\LK}-proof is called {\em regular} if each two $\forall$-inferences have
different eigenvariables and different from the free variables in the conclusion of the proof.
From now on we assume w.l.o.g.\ that all {\LK}-proofs are regular.

\subsection{From Sequent Calculus to Expansion Proofs}\label{complete}
In this section we describe how to read off expansion trees from {\LK}-proofs (with sequents as multisets), thus
obtaining a completeness theorem for expansion proofs. For representing
a formula $A$ that is introduced by (implicit) weakening we use the natural
coercion of $A$ into an expansion tree, denoted by $A^\mathrm{E}$: $(\exists x A)^{E}=\exists x A +^{x} A^{E}$, $(\forall x A)^{E}=\forall x A +^{\alpha} A^{E}$ ($\alpha$ fresh), $(E_{1}\circ E_{2})^{E}=E_{1}^{E}\circ E_{2}^{E}$, $L^{E}=L$ for $L$ atomic.
%
%we define the following
%mapping from formulas to expansion trees:
%\begin{align*}
%L^\mathrm{E} &= L\ \mbox{for a literal}\ L,\\
%(A\circ B)^\mathrm{E} &= A^\mathrm{E}\circ B^\mathrm{E}\ \mbox{for $\circ\in\{ \land, \lor \}$},\\
%(\forall x\, A)^\mathrm{E} &= \forall x\, A +^\alpha A^\mathrm{E}\ \mbox{for a fresh variable $\alpha$, and}\\
%(\exists x\, A)^\mathrm{E} &= \exists x\, A\ \mbox{with the empty set of expansions}.
%\end{align*}
For a sequent $\Gamma = A_1,\ldots,A_n$ we define
$\Gamma^\mathrm{E} = A_1^\mathrm{E},\ldots, A_n^\mathrm{E}$.

\begin{definition}
For an {\LK}-proof $\pi$ define the expansion proof $\Exp(\pi)$ by
induction on $\pi$:
\begin{enumerate}
\item If $\pi$ is an initial sequent $\Gamma, A, \dual{A}$, thus with $A$ atomic, then
$\Exp(\pi) = \Gamma^\mathrm{E}, A, \dual{A}$

\item If 
$
\pi
=
\begin{array}{c}
\lkand{\Gamma, A \land B}{
  \deduce{\Gamma, A}{(\pi_A)}
  &
  \deduce{\Gamma, B}{(\pi_B)}
}
\end{array}
$
with $\Exp(\pi_A) = \calP_A, E_A$ and $\Exp(\pi_B) = \calP_B, E_B$ where
$\Shallow(E_A) = A$ and $\Shallow(E_B) = B$, then
$\Exp(\pi) = (\calP_A, \calP_B, E_A \land E_B)$.

\item If
$
\pi
=
\begin{array}{c}
\lkor{\Gamma, A\lor B}{
  \deduce{\Gamma, A, B}{(\pi')}
}
\end{array}
$
with $\Exp(\pi') = \calP, E_A, E_B$ where $\Shallow(E_A) = A$ and $\Shallow(E_B) = B$,
then $\Exp(\pi) = (\calP, E_A \lor E_B)$.

\item If
$
\pi
=
\begin{array}{c}
\lkforall{\Gamma, \forall x\, A}{
  \deduce{\Gamma, A\unsubst{x}{\alpha}}{(\pi_A)}
}
\end{array}
$
with $\Exp(\pi_A) = \calP, E$ where $\Shallow(E) = A\unsubst{x}{\alpha}$,
then $\Exp(\pi) = (\calP, \forall x\, A +^\alpha E)$.

\item If
$
\pi
=
\begin{array}{c}
\lkexists{\Gamma, \exists x\, A}{
  \deduce{\Gamma, \exists x\, A, A\unsubst{x}{t}}{(\pi_A)}
}
\end{array}
$
with $\Exp(\pi_A) = \calP, E, E_t$ where $\Shallow(E) = \exists x\, A$ and
$\Shallow(E_t) = A\unsubst{x}{t}$, then $\Exp(\pi) = (\calP,
E, \exists x\, A +^t E_t)$.

\item If
$
\pi
=
\begin{array}{c}
\lkcut{\Gamma}{
  \deduce{\Gamma, A}{(\pi^+)}
  &
  \deduce{\dual{A}, \Gamma}{(\pi^-)}
}
\end{array}
$
for $A$ positive with $\Exp(\pi^+) = \calP^+, E^+$ and $\Exp(\pi^-) = \calP^-, E^-$
where $\Shallow(E^+) = A$ and $\Shallow(E^-) = \dual{A}$, then $\Exp(\pi) = ((E^+, E^-), \calP^+,  \calP^-)$.
\end{enumerate}
\end{definition}

%Note that the behavior of the above definition of $\Exp(\cdot)$ on binary rules
%is to merge expansions of both subproofs (including cuts). This is the
%reason for the relationship between sequent calculus proofs and expansion
%proofs which on the one hand are strongly connected
%structurally~\cite{Chaudhuri12Systematic,Chaudhuri16Multi} but at
%the same time have different complexity~\cite{Baaz12Complexity}.
%
\begin{theorem}[Completeness]\label{thm:completeness}
If $\pi$ is an {\LK}-proof of a sequent $\Gamma$, then $\Exp(\pi)$ is an
expansion proof such that $\Shallow(\pi)=\Gamma$. If $\pi$ is cut-free, then so is $\Exp(\pi)$.
\end{theorem}%
\begin{proof}
That $\Exp(\pi)$ satisfies weak regularity  and the eigenvariable condition follows directly from the
definitions as we are dealing with regular {\LK}-proofs only, thus we are constructing regular expansion proofs as well.
By a straightforward induction on $\pi$ one shows that $\Deep(\Exp(\pi))$
is a tautology.
Acyclicity is also shown inductively by observing that
if $\alpha$ is a free variable in the end-sequent of $\pi$,
then $\alpha$ is not an eigenvariable in $\Exp(\pi)$. This
implies that if $w$ is the new expansion introduced in
the construction of $\Exp(\pi)$, and $v$ is an old expansion
in $\Exp(\pi)$, then $w\not>v$, which in turn yields
acyclicity.
\end{proof}
\subsection{From Expansion Proofs to Sequent Calculus}\label{sound}
In this section we show how to construct an {\LK}-proof (with sequents as sets) from a given expansion
proof. To this aim we introduce a calculus {\LKE}, generalizing the treatment
in~\cite{Miller87Compact}, that works on  sequences of expansion
trees and cuts instead of sequents of formulas.
\begin{definition}
The axioms of {\LKE} are of the form $\calP, A, \dual{A}$ for an atom $A$.
The inference rules are
\[
\lkforall{\calP, \forall x\, A +^\alpha E_1, \ldots,  \forall x\, A +^\alpha E_n }{\calP, E_1, \ldots, E_{n}}
\]
\[
\lkexists{\calP, \exists x\, A +^{t_1} E_1\cdots+^{t_{i}} E_{i}\cdots +^{t_n} E_n}{
  \calP, \exists x\, A +^{t_1} E_1 \cdots +^{t_n} E_n +^{t_{i}}E_{i}
}
\]
\[
\lkexists{\calP, \exists x\, A +^{t_1} E_1 \cdots +^{t_n} E_n, \ldots, \exists x\, A +^{s_1} F_1 \cdots +^{s_{m}} F_m}{
  \calP, \exists x\, A +^{t_1} E_1 \cdots +^{t_{i}} E_i, \ldots, \exists x\, A +^{s_1} F_1 \cdots +^{s_{j}} F_j, E_{i+1} \ldots, E_{n}, \ldots, F_{j+1}, \ldots, F_{m}
}
\]
\[
\lkand{\calP, E_1\land E_2}{\calP, E_1 & \calP, E_2}
\quad
\lkor{\calP, E_1\lor E_2}{\calP, E_1, E_2}
\quad
\lkcut{\calP, \{E_1,  F_1\},\ldots, \{E_{n}, F_n\}}{\calP, E_1, \ldots, E_{n} & \calP, F_{1},\ldots, F_{n}}
\]
with the following side conditions: %the $\forall$-expansion $+^{\alpha}$ does not occur in $\calP$ and 
$\Shallow(E_{1})=\ldots =\Shallow(E_{n})$ for the cut rule; $t_{i+1}=\ldots=t_{n}=\ldots=s_{j+1}=\ldots=s_{m}$  for the second $\exists$ rule;
 the eigenvariable condition for the $\forall$ rule: $\alpha$ must not occur in
$\Shallow(\calP), \forall x\, A$.  %and the $\forall$-expansion $+^{\alpha}$ does not occur in $\calP$.
\end{definition}
The reader is invited to note that $\Shallow(\calP), \forall x\, A $ does
{\em not} include the cut formulas of $\calP$, though they may -- and indeed often
have to -- contain the eigenvariable
$\alpha$. %Furthermore, it should be kept in mind that the expansion terms at the $\exists$-rule form a {\em set}, i.e.\ the above rule allows to take any instance as there is no such thing as a last or rightmost instance. 
An important
feature of the above calculus, which is easily verified, is that if $\pi$ is an {\LKE}-proof,
then $\Shallow(\pi)$ -- defined as the result of replacing  in $\pi$ each sequence  $\calP$ of expansion trees and cuts  with $\Shallow(\calP)$ -- is a {\LK}-proof. In the following proof we describe how
to transform expansion proofs to {\LK}-proofs.
\begin{theorem}[Soundness]\label{thm:soundness}
If $\calP$ is an expansion proof of a sequent $\Gamma$, then there
is an {\LK}-proof of $\Gamma$. If $\calP$ is cut-free, then so is the {\LK}-proof.
\end{theorem}
\begin{proof}
It is enough to construct an {\LKE}-proof $\pi$ of $\calP$, as then $\Shallow(\pi)$ is
a proof of $\Shallow(\calP) = \Gamma$. The construction will be carried out
by induction on the number of nodes in $\calP$. The inductive statement we are going to prove  is: if $\calP$ is an expansion proof, then there is a $\LKE$-proof  $\pi$ of $\calP$.

If $\calP$ contains only literals, the thesis is obvious.

If $\calP = \calP', E_1 \lor E_2$ for some $\calP'$, $E_1$ and $E_2$,
then  $\calP', E_1, E_2$ is a strictly smaller expansion proof.
By the induction hypothesis we obtain an {\LKE}-proofs $\pi'$ of $\calP', E_1, E_2$
from which a proof of $\calP$ is obtained by an $\lor$-inference.
For $\calP = \calP', E_1 \land E_2$, proceed analogously.

So assume there are no top-level conjunctions or disjunctions. We observe that for any non-top-level quantifier expansion there is some top-level quantifier expansion that dominates it, and is smaller than it, according to $<_\calP$. Thus, the $<_\calP$-minimal quantifier expansions  are all top-level. By the acyclicity of $<_\calP$ there must be a $<_\calP$-minimal quantifier expansion or cut. In the case the $<_\calP$-minimal expression is a quantifier expansion, then it is a top-level one.

 For
the case of cut, we proceed as follows: let  $C=\{ E_1, \forall x\, A +^\alpha F_1 \}$ be a cut minimal with respect to $<_\calP$ (if $\Shallow(C)$ does not begin with a quantifier the argument is easier). Then $\forall x\, A, A[\alpha/x]\in \Br(C), s$, which, by weak regularity of $\calP$, forces every element of $\calP$ containing $+^{\alpha}$ to be a cut of the shape $\{E_i, \forall x\, A +^\alpha F_i\}$.  Then, 
 we can write 
$$\calP=\{E_1, \forall x\, A +^\alpha F_1\},\ldots, \{E_{n},\forall x\, A +^\alpha F_n\},\calP'$$
where $+^{\alpha}$ does not occur in $\calP'$. Now, 
$$\Deep(\calP)=(\Deep(E_{1})\land \Deep(F_{1}))\lor \ldots \lor (\Deep(E_{n})\land \Deep(F_{n}))\lor \Deep(P')$$
Therefore, 
$$\Deep(E_1,\ldots, E_{n}, \calP')=\Deep(E_{1})\lor\ldots\lor \Deep(E_{n})\lor \Deep(\calP')$$
$$\Deep(\forall x\, A +^\alpha F_1,\ldots,\forall x\, A +^\alpha F_{n}, \calP')=\Deep(F_{1})\lor\ldots\lor \Deep(F_{m})\lor \Deep(\calP')$$ are tautologies.  To prove weak regularity of $E_1,\ldots, E_{n}, \calP'$ and $\forall x\, A +^\alpha F_1,\ldots, \forall x\, A +^\alpha F_{n}, \calP'$, we observe that $\Br(E_1,\ldots, E_{n}, \calP')$ and $\Br(\forall x\, A +^\alpha F_1,\ldots, \forall x\, A +^\alpha F_{n}, \calP')$ are contained in $\Br(\calP)$; the only issue is when a branch belongs to $\Br(\forall x\, A +^\alpha F_i)$.  
 In this case, in order to show weak regularity, we have to show that for every branch $b = s, \forall x\, A, A[\alpha/x], s'$, we have that $s$ is empty  and the branch $b$ is the branch of a tree. By weak regularity of $\calP$, the same branch $b$ was in $\calP$ a branch of a cut. Thus, it is enough to observe that $b$ belongs to some cut $\{E_{j}, \forall x\,A +^\alpha F_j\}$, for some $j=i,\dots,n$, otherwise $b$ would belong to some cut in $\calP'$,  impossible by construction of $\calP'$.
 %because the branch was in $\calP$ a branch of a cut; but no such branch can belong to a cut in $\calP'$, by construction of $\calP'$.  
 Furthermore,  the orderings of the expansions and cuts of $E_1,\ldots, E_{n}, \calP'$ and $\forall x\, A +^\alpha F_1,\ldots, \forall x\, A +^\alpha F_{n}, \calP'$ are suborderings of $<_{\calP}$, hence also acyclic. Last, $\Shallow(E_1,\ldots, E_{n}, \calP')$ and $\Shallow(\forall x\, A +^\alpha F_1,\ldots,\forall x\, A +^\alpha F_{n}, \calP')$ contain the same free variables of $\Shallow(\calP)$ plus those of $\forall x A$; now, no $\forall$-expansion $+^{\beta}$ of $\calP$ can have $\beta$ occur in $\forall x A$, otherwise  $+^{\beta}<_{\calP}C$, contradicting the minimality assumption on $C$, so we have that the eigenvariable condition holds.  Then, by
the induction hypothesis we obtain {\LKE}-proofs $\pi_1,\pi_2$ of $E_1,\ldots, E_{n}, \calP'$
and $\forall x\, A +^\alpha F_1,\ldots, \forall x\, A +^\alpha F_{n}, \calP'$ respectively, from which a proof of $\calP$ is obtained
by a cut.

For the case of the minimal node being an $\exists$-expansion, let  $ \exists x\, A\cdots +^t E\cdots $ be an expansion tree of $\calP$ such that $+^{t}$ is minimal with respect to $<_\calP$.  As we said, $+^t$ occurs at top-level. We move all top-level $+^t$  at the end of the lists of $\exists$-espansions relative to the corresponding top level formula. In this way, we can rewrite $\calP$ as
$$\exists x\, A +^{t_1} E_1 \cdots +^{t_{n}} E_n, \ldots, \exists x\, A +^{s_1} F_1 \cdots +^{s_{m}} F_m, \calP'$$
in such a way that: $\Shallow(E)\neq \exists x A$ for every expansion tree $E$ in $\calP'$; there are $i, \ldots, j$ such that $t_{i}=\ldots=t_{n}=t, \ldots, s_{j}=\ldots=s_{m}=t$ and $i'<i, \ldots, j'<j$ implies $t_{i'}\neq t, \ldots, s_{j'}\neq t$.
Let 
$$\calQ=\calP', \exists x\, A +^{t_1} E_1 \cdots +^{t_{i}} E_i, \ldots, \exists x\, A +^{s_1} F_1 \cdots +^{s_{j}} F_j, E_{i+1} \ldots, E_{n}, \ldots, F_{j+1}, \ldots, F_{m}$$
Then
$\Deep(\calP)=\Deep(\calQ)$, so they are both tautologies.   To prove weak regularity of $\calQ$, we first observe that every  $s\in\Br(\calQ)$ is either already contained in $\Br(\calP)$ or $s\in\Br(E_{k})$ or $s\in\Br(F_{h})$, with $i+1\leq k\leq n$ and $j+1\leq h\leq m$ and  $\exists x A, s$ in $\Br(\calP)$. Thus the only problematic case  is when $b_{1}$ belongs to $ \Br(\calP)$ but not to the branches of the new trees of $\calQ$, while, for instance, $b_{2}\in \Br(E_{k})$, with $i+1\leq k\leq n$. We show that it cannot be the case that $b_{1}=s, \forall y B, B[\alpha/y], s'$ and $b_{2}=r, \forall y C, C[\alpha/y], r'$: assume for the sake of contradiction that it is. Since $\exists x A, r, \forall y C, C[\alpha/y], r'\in \Br(\calP)$, we have $s=\exists x A, r$, by weak regularity of $\calP$. Therefore, $b_{1}\notin\Br(\calP')$, so without loss of generality, say $s=\exists x A, A[s_{l}/x], s''$. But $r=A[t_{k}/x], r''$ and since $s_{l}\neq t= t_{k}$ by construction, we have a contradiction.
\noindent Furthermore,  the orderings of the expansions and cuts of  $ \calQ$ are suborderings of $<_{\calP}$, hence also acyclic. Last, no $\forall$-expansion $+^{\beta}$ of $\calQ$ can have $\beta$ occur in $A[t/x]$, otherwise either $\beta$ occurs in $t$ or  $\beta$ already occurs in $\exists x A$, against the assumptions  on minimality of $+^{t}$, in the first case, agains the assumption on $\calP$ having the eigenvariable condition, in the second case. 
 Then, by
the induction hypothesis we obtain a {\LKE}-proof $\pi$ of $\calQ$, from which a proof of $\calP$ is obtained
by  the second $\exists$ rule and a  number of applications of the first $\exists $ rule, taking care of the rewriting of $\calP$ that we made.

For the case of the minimal node being a $\forall$-expansion, let  $ \forall x\, A +^\alpha E $ be an expansion tree of $\calP$ such that $+^{\alpha}$ is minimal with respect to $<_\calP$.  As we said, $+^\alpha$ occurs at top-level. Then $\forall x\, A, A[\alpha/x], s\in \Br( \forall x\, A +^\alpha E )$, which, by weak regularity of $\calP$, forces every element of $\calP$ containing $+^{\alpha}$ to be  an expansion tree of the shape $\forall x\, A +^{\alpha} F$.  Then, 
 we can write 
$$\calP=\forall x\, A +^\alpha F_1,\ldots, \forall x\, A +^\alpha F_n,\calP'$$
where $+^{\alpha}$ does not occur in $\calP'$. Now, 
$\Deep(\calP)=\Deep(F_1,\ldots,  F_n,\calP')$, so they are both tautologies.   To prove weak regularity of $F_1,\ldots,  F_n,\calP'$, it is enough to note that every  $s\in\Br( F_1,\ldots, F_{n}, \calP')$ is either already contained in $\Br(\calP)$ or $s\in F_{k}$, for $1\leq k\leq n$ and $\forall x A, s$ is in $\Br(\calP)$. Thus the only problematic case  is when $b_{1}$ belongs to  $\Br(\calP')$ but not to the branches of the new trees of $\calQ$, while, for instance, $b_{2}\in \Br(F_{k})$, with $1\leq k\leq n$. We show that it cannot be the case that $b_{1}=s, \forall y B, B[\beta/y], s'$ and $b_{2}=r, \forall y C, C[\beta/y], r'$: assume for the sake of contradiction that it is. We first notice that $r=A[\alpha/x], r''$.
 Moreover, since $\forall x A, r, \forall y C, C[\beta/y], r'\in \Br(\calP)$, we have $s=\forall x A, r$, by weak regularity of $\calP$. Therefore, $b_{1}\notin\Br(\calP')$ and we have a contradiction.
Furthermore,  the orderings of the expansions and cuts  $ F_1,\ldots,  F_{n}, \calP'$ are suborderings of $<_{\calP}$, hence also acyclic. Last, no $\forall$-expansion $+^{\beta}$ of $F_1,\ldots,  F_n,\calP'$ can have $\beta$ occur in $A[\alpha/x]$, otherwise either $\beta=\alpha$ or  $\beta$ already occurs in $\forall x A$, against the assumptions  on $+^{\alpha}$ not occurring in $\calP'$ or against weak regularity of $\calP'$, in the first case, against the assumption on $\calP$ having the eigenvariable condition, in the second case. 
 Then, by
the induction hypothesis we obtain  a {\LKE}-proof $\pi$ of $F_1,\ldots, F_{n},\calP'$, from which a proof of $\calP$ is obtained
by  the $\forall$ rule, because ${\alpha}$ does not occur in $\Shallow(\calP)$.

\end{proof}
%
%\begin{definition} The {\LK}-proof constructed in the above proof will be called $\Seq(\calP)$. \end{definition}

\section{Cut-Elimination}\label{sec:cutelim}
Given any expansion proof $\calP$ there is always a cut-free expansion proof of $\Shallow(\calP)$: by the soundness Theorem \ref{thm:soundness}, transform $\calP$ into an $\LK$- proof, then perform Gentzen cut-elimination and obtain a cut-free proof; finally map it back to a cut-free expansion proof by the completeness Theorem \ref{thm:completeness}. Nevertheless, the interest of expansion proofs is that they allow to investigate the combinatorics and the computational meaning of cut-elimination, with the additional advantage of factoring out tedious structural rules such as cut-permutations.
In this section we indeed define a natural reduction system for expansion
proofs, such that the normal forms are cut-free expansion proofs.
We prove weak normalization and discuss the status of other properties such
as strong normalization and confluence in comparison to other systems from the literature. 

\subsection{Cut-Reduction Steps}\label{sec:cut_reduction_steps}
In the following, by a substitution $\sigma$ we mean as usual a finite map from variables to first-order terms, and if $e$ is any syntactic expression, $e\sigma$ denote the expression resulting from $e$ after simultaneous replacement of each variable $x$ in the domain of $\sigma$ with $\sigma(x)$.
To make
sure that the application of a substitution transforms expansion trees  into expansion trees  we restrict the set of permitted
substitutions: a substitution $\sigma$ can only be applied to an expansion tree
 $E$, if it acts on the eigenvariables of $E$ as a renaming, more
precisely: if $+^{\alpha}$ is a $\forall$-expansion of $E$, then  $\sigma(\alpha)$ is a variable.
%\footnote{TODO: what about regularity ? - where should it be ensured? here by
%additional condition?}. 
Otherwise $\sigma$ would destroy the $\forall$-expansions.
%
%\begin{definition}\label{def.cutreduction}

While  presenting our cut-reduction steps, we have to take into account weak regularity: cut-reduction will {\em duplicate}
sub-proofs, making it necessary to discuss the renaming of variables, as
in the case of the sequent calculus. We will carefully indicate,
in the case of a duplication, which subtrees should be subjected to a
variable renaming and which variables are to be renamed.

%To this end, for an expansion proof
%$\calP$ we introduce the (imprecise) notation $\calP^r$ to stand for 
%$\calP\sigma$, where $\sigma$ is a renaming establishing regularity. It will
%always be clear from the context how to construct an appropriate $\sigma$.
%
%TODO: only horizontal arrows (everywhere)
%
\begin{definition}[Cut-Reduction Steps]
The cut-reduction steps, relating expansion proofs $\calQ,\calQ'$ and
written $\calQ\mapsto\calQ'$,
are:\\ \\
\textbf{Quantifier Step}
\begin{small}
$$\{  \exists x\, A +^{t_1} E_1 \cdots +^{t_n} E_n , \forall x\, \bar{A} +^{\alpha_{1}} F_{1} \}, \ldots, \{  \exists x\, A +^{t_{p}} E_p \cdots +^{t_{l}} E_l , \forall x\, \bar{A} +^{\alpha_{q}} F_{q} \}, \calP$$
$$\mapsto$$ $$\{ E_1, F_{1}\eta_1\sigma_{1}\},\ldots, \{ E_1, F_{q}\eta_1\sigma_{1}\},\ldots, \{E_{l}, F_{1}\eta_l\sigma_{l} \},\ldots,   \{E_{l}, F_{q} \eta_l\sigma_{l} \}, \calP, \calP\eta_1\sigma_{1}, \ldots, \calP\eta_l  \sigma_{l}
$$
\end{small}where: $\sigma_1=[{t_1}/\alpha_{1}\dots t_{1}/\alpha_{q}], \ldots,\sigma_{l}= [{t_l}/\alpha_{1}\dots t_{l}/\alpha_{q}]$; the $\forall$-expansions $+^{\alpha_{1}}, \dots, +^{\alpha_{q}}$ do not occur in $\calP$; no cut in $\calP$ has shallow formula $\exists x A $; $\eta_{1}, \dots, \eta_{l}$ are renamings to fresh variables of the eigenvariables $\beta$ of $\calP, F_{1},\ldots, F_{q}$ such that for some $1\leq i\leq q$ and occurrence of $+^{\alpha_{i}}$ we have    $+^{\alpha_{i}}<_{\calP_{1}}+^{\beta}$.
\\ \\
\textbf{Propositional Step}
\[
  \{ E_1 \lor F_1, E'_1\land F'_1 \},\dots, \{ E_m \lor F_m, E'_m\land F'_m \}, \calP\ \mapsto\ \{ E_1, E'_1 \}, \{ F_1, F'_1 \},\dots, \{ E_m, E'_m \}, \{ F_m, F'_m \}, \calP
\]
where $\Shallow(E_1 \lor F_1)=\dots=\Shallow(E_m \lor F_m)$ and for every $C$ in $\calP$, $\Shallow(C)\neq \Shallow(E_1 \lor F_1)$.

\textbf{Atomic Step}
\[
\{ A, \dual{A} \}, \calP\ \mapsto\ \calP \quad \mbox{for an atom $A$.}
\]
\end{definition}
These reduction rules are very natural:
atomic cuts are simply removed and propositional cuts are decomposed. 
The reduction of a quantified cut is, when thinking about cut-elimination in the
sequent calculus, intuitively  appealing:  existential cuts are replaced
by  cuts on a disjunction of the instances. The fact that at least one of these rules that can be applied to any expansion proof containing cuts will be proved in Theorem \ref{thm:weak_normalization}. 

\noindent One may think that the quantifier reduction rule already incorporates a reduction strategy, because several cuts are reduced in parallel. However, a strategy implies a choice and there is no real choice here: when the main eigenvariable occurs in other cuts,  all these cuts have to be regarded as linked together, otherwise reducing one of them would destroy the soundness of the others. Moreover, all the cuts with the same shallow formula must be reduced, otherwise weak regularity would not be preserved.

 \noindent The reason why only the eigenvariables greater than some $\alpha_{i}$  are renamed is that these are the variable indirectly affected by the substitutions $[t_{i}/\alpha]$.  Semantically, the witnesses that these variables represent are influenced by the substitutions, so for each of them a new collection of eigenvariables is created.

%---- The reason why only the variables greater than $\alpha$  are renamed is that we are applying the renaming only to one side of the new generated cuts, but obviously we want them to be on dual formulas: in the proof of Lemma \ref{lem:red_pres_proof}, we shall show that $\exists x A$ indeed does not contain renamed variable. ----

\noindent One surprising aspect of the quantifier-reduction rule is the presence of $\calP$, without
a substitution applied, on the right-hand side of the rule: in general, $\calP$ will contain $\alpha$, and
one would expect that occurrences of $\alpha$ are redundant (since $\alpha$ is ``eliminated''
by the rule). The reason why this occurrence of $\calP$ must be present is that $\alpha$ is not,
in fact, eliminated since some $t_i$ might contain it. This situation occurs, for example, when
translating from a regular $\LK$-proof where an $\exists$-quantifier may be instantiated by any
term, and we happen to choose an eigenvariable from a different branch of the proof. In the 
sequent calculus, this situation can in principle be avoided by using a different witness for the
$\exists$-quantifier, but realizing such a renaming in expansion proofs is technically non-trivial
due to the global nature of eigenvariables. For simplicity of exposition, we therefore allow this
somewhat unnatural situation and leave a more detailed analysis for future work. Note that this $n+1$-st
copy is reminiscent of the duplication behavior of the $\varepsilon$-calculus~\cite{Hilbert39Grundlagen2},
see~\cite{Moser06Epsilon} for a contemporary exposition in English.
\begin{remark}[On Bridges]
We note that this phenomenon also occurs in the proof forests of~\cite{Heijltjes10Classical},
where it is an example of {\em bridge}. Bridges, when ignored, can generate cycles in the dependency relation $<_{\calQ}$. In \cite{Heijltjes10Classical},
they are addressed with a {\em pruning reduction} that eliminates them and the weak normalization proof of
that system depends on this pruning. In our setting, we do not need additional machinery
for proving weak normalization (see Section~\ref{sec:weak_normalization}). The reason is our renaming policy: while in \cite{Heijltjes10Classical} every occurrence of every variable above $\alpha$ in the dependency relation $<_{\calQ}$ is renamed, in our case only some of those occurrences are renamed, namely those which are not in $t_{1}, \ldots, t_{n}$ or in $E_{1}, \ldots, E_{n}$. In such a way, bridges are broken by our cut-reduction step, so that cycles in the dependency relation cannot be generated. Furthermore,
the counterexample to strong normalization from~\cite{Heijltjes10Classical} also contains
a bridge; we investigate (a translation of) this counterexample in
Section~\ref{sec:strong_normalization} and find that it is not a counterexample in our setting for the reasons explained.
\end{remark}
\begin{example} We now consider an example of cut reduction steps, in particular when an eigenvariable $\gamma$ occurs more than once. 

$$\{ \forall x \exists y P(x,y) +^\alpha \exists y P(\alpha,y) +^{f(\alpha, \beta)} P(\alpha, f(\alpha,\beta)), \exists x\forall y\overline{P}(x,y) +^q \forall y \overline{P}(q,y) +^\gamma \overline{P}(q, \gamma)\}, $$

Ê$$\{ \forall x\exists y P(x,y) +^{\beta} \exists y P(\beta,y) +^{g(\alpha, \beta)} P(\beta, g(\alpha,\beta))  , \exists x \forall y\overline{P}(x,y) +^q \forall y \overline{P}(q,y) +^\gamma \overline{P}(q,\gamma)\}$$

$$\mapsto$$

$$ \{\exists y P(q,y) +^{f(q, q)} P(q, f(q,q)),  \forall y\overline{P}(q,y) +^\gamma \overline{P}(q, \gamma)\},$$

$$ \{\exists y P(q,y) +^{g(q, q)} P(q, g(q,q)),  \forall y\overline{P}(q,y) +^\gamma \overline{P}(q, \gamma)\},$$

$$ \{\exists y P(q,y) +^{f(q, q)} P(q, f(q,q)),  \forall y\overline{P}(q,y) +^\gamma \overline{P}(q, \gamma)\},$$

$$ \{\exists y P(q,y) +^{g(q, q)} P(q, g(q,q)),  \forall y\overline{P}(q,y) +^\gamma \overline{P}(q, \gamma)\},$$

We see above two identical pairs of cuts, due to the first expansion containing two times the same tree. Reducing this last expansion, we obtain $8$ occurrences of  $\{P(q, f(q,q)), \overline{P}(q, f(q,q)\}$ and $\{P(p, f(p,p)), \overline{P}(p, f(p,p)\}$.

\end{example}

A simple property we are going to need is that substitution commutes with $\Deep(\cdot)$, $\Shallow(\cdot)$ and $\Br(\cdot)$.
\begin{lemma}\label{lemma-subs}
For every  substitution $\sigma$, $$\Shallow(E\sigma)=\Shallow(E)\,\sigma$$ 
$$\Deep(E\sigma)=\Deep(E)\,\sigma$$ 
$$\Br(E\sigma)=\{s\,\sigma\ |\ s\in\Br(E)\}$$
\end{lemma}
\begin{proof}
By a straightforward induction on $E$.
\end{proof}

We now prove that the cut-reduction relation is sound.

\begin{lemma}[Soundness of Cut-Reduction]\label{lem:red_pres_proof}
If $\calP_{1}\mapsto\calP_{2}$ and $\calP_{1}$ is an expansion proof, then $\calP_{2}$ is an expansion
proof. Furthermore, $\Shallow(\calP_{1})=\Shallow(\calP_{2})$.
\end{lemma}
\begin{proof}
We only give the proof for the quantifier cut-reduction step; the proof for the other reduction steps
is analogous and simpler.
Let $\sigma_1, \ldots, \sigma_{l}$ be respectively the substitutions $[{t_1}/\alpha_{1}\dots t_{1}/\alpha_{q}], \ldots, [{t_l}/\alpha_{1}\dots t_{l}/\alpha_{q}]$ and
assume 
\begin{small}
$$\calP_{1}=\{  \exists x\, A +^{t_1} E_1 \cdots +^{t_n} E_n , \forall x\, \bar{A} +^{\alpha_{1}} F_{1} \}, \ldots, \{  \exists x\, A +^{t_{p}} E_p \cdots +^{t_{l}} E_l , \forall x\, \bar{A} +^{\alpha_{q}} F_{q} \}, \calP$$
$$\mapsto$$ $$\{ E_1, F_{1}\eta_1\sigma_{1}\},\ldots, \{ E_1, F_{q}\eta_1\sigma_{1}\},\ldots, \{E_{l}, F_{1}\eta_l\sigma_{l} \},\ldots,   \{E_{l}, F_{q} \eta_l\sigma_{l} \}, \calP, \calP\eta_1\sigma_{1}, \ldots, \calP\eta_l  \sigma_{l}
$$
$$=\calQ$$
\end{small}where the $\forall$-expansions $+^{\alpha_{1}}, \dots, +^{\alpha_{q}}$ do not occur in $\calP$, no cut in $\calP$ has shallow formula $\exists x A $ and $\eta_{1}, \dots, \eta_{l}$ are renamings to fresh variables of the eigenvariables $\beta$ of $\calP, F_{1},\ldots, F_{q}$ such that for some $1\leq i\leq q$ and occurrence of $+^{\alpha_{i}}$ we have    $+^{\alpha_{i}}<_{\calP_{1}}+^{\beta}$. We observe that no $\alpha_{k}$ is in the domain of any $\eta_{j}$, otherwise for some $i$ we would have $+^{\alpha_{i}}<_{\calP_{1}}+^{\alpha_{k}}$, which is only possible when $+^{\alpha_{i}}<_{\calP_{1}}+^{\gamma} <_{\calP_{1}} C_{k}$ where  $\gamma$ is an  eigenvariable and $C_{k}$ is the $k$-th of the displayed cuts of $\calP_{1}$. But then $\gamma$ would occur in $\exists x A$, hence also $+^{\alpha_{i}}<_{\calP_{1}}+^{\alpha_{i}}$, which contradicts the acyclicity of $<_{\calP_{1}}$. We now show that no variable $\gamma$ in the domain of any $\eta_{k}$ can occur in any $\overline{A}[\alpha_{j}/x]=\Shallow(F_{j})$: assume by contradiction that $+^{\alpha_{i}} <_{\calP_{1}} +^{\gamma}$ and $\gamma$ occurs in $\overline{A}[\alpha_{j}/x]$. Then, letting $C_{i}$ be the $i$-th of the displayed cuts of $\calP_{1}$, we immediately  have by Definition \ref{defi-dependency} the contradiction $C_{i}<_{\calP_{1}}+^{\alpha_{i}}<_{\calP_{1}}+^{\gamma}<_{\calP_{1}}C_{i}$, the last relation due to $\gamma\neq \alpha_{j}$ as shown above, thus $\gamma$ occurring in $\forall x \overline{A}$ and thus in $\Shallow(C)=\exists x A$. Therefore, by Lemma \ref{lemma-subs}, for every $1\leq h\leq q$, we have $\Shallow(F_{h}\eta_{i}\sigma_{i})=\Shallow(F_{h})\eta_{i}\sigma_{i}= \Shallow(F_{h})\sigma_{i}=\overline{A}[t_{i}/x]$, the last equality due to $\alpha_{j}$ for $j\neq i$ not occurring in $\overline{A}$, otherwise $+^{\alpha_{i}}<_{\calP_{1}}+^{\alpha_{j}}$. We have thus shown that the displayed cuts of $\calQ$ are indeed between dual formulas.

We now prove several properties, namely that $\calQ$ is weakly regular, $\Deep(Q)$ is valid, $<_{\calQ}$ is acyclic and $\Shallow(\calQ)$ does not contain eigenvariables, which means that $\calQ$ is an expansion proof.

\textbf{Weak Regularity}. %First, using Lemma \ref{lemma-v}, we want to show that $$\calQ=  \{ E_{1}, E\eta_1\sigma_1\}\omerge\ldots\omerge\{E_{n},E\eta_n\sigma_n \} \omerge \calP\omerge\calP\eta_1\sigma_{1}\omerge\ldots \omerge \calP\eta_n\sigma_{n}$$ is weakly regular, from which one can conclude by Lemma \ref{lemma-mergesound} that  $\calP_{2}$ is regular. 
Letting $\eta_{0}$ and $\sigma_{0}$ be the empty substitution, by Lemma \ref{lemma-subs} any branch in $\Br(\calQ)$ is of the form $p\eta_{i}\sigma_{i}$, where $i\in\{0, 1, \ldots, n\}$ and $$p\in \mathsf{B}:=\Br(\calP)\cup \Br(E_{1})\cup\ldots \cup \Br(E_{l})\cup\ldots \cup \Br(F_{1})\cup\ldots \cup \Br(F_{q})$$
Thus, for each $p\in \mathsf{B}$, either $p\in \Br(\calP)$ or $\forall x \overline{A}, p\in\Br(\calP_{1})$ or $\exists x {A}, p\in\Br(\calP_{1})$.
 Assume
$$s \eta_{i}\sigma_{i}, \forall x B\, \eta_{i}\sigma_{i}, B[\beta/x]\eta_{i}\sigma_{i}, s' \eta_{i}\sigma_{i} \in \Br(S)$$
$$r \eta_{j}\sigma_{j}, \forall x C\,\eta_{j}\sigma_{j}, C[\delta/x] \eta_{j}\sigma_{j}, r'\eta_{j}\sigma_{j}\in \Br(R) $$
with $S$ and $R$ in $\calQ$. Let $\overline{s}=s, \forall x B, B[\beta/x], s'$ and $\overline{r}=r, \forall x C, C[\delta/x], r'$. We first rule out that either $\overline{s}\in\mathsf{B}$ and $Jx A, \overline{r}\in \mathsf{B}$ or $Jx A, \overline{s}\in\mathsf{B}$ and $\overline{r}\in \mathsf{B}$, with $J\in\{\forall, \exists\}$. Indeed, since in the first case $Jx A, \overline{r}\in \mathsf{B}$ and in the second case $Jx A, \overline{s}\in \mathsf{B}$ belong to the branches of one of the displayed cuts of $\calP_{1}$,   by weak regularity of $\calP_{1}$ we would have that $\overline{s}$ or $\overline{r}$ belong to the branches of a cut in $\calP$ with shallow formula $\exists x A$, contrary to our assumptions.
%and: $s, \forall x B, B[\beta/x], s'\in \calP_{1}$ or $\exists x A, s, \forall x B, B[\beta/x], s'\in \calP_{1}$\\ or $\forall x \overline{A}, s, \forall x B, B[\beta/x], s'\in \calP_{1}$;  $r, \forall x C, C[\delta/x], r'\in \calP_{1}$ or $\exists x A, r, \forall x C, C[\delta/x], r'\in \calP_{1}$ or $\forall x \overline{A}, r, \forall x C, C[\delta/x], r'\in \calP_{1}$

Now there are two cases.
\begin{itemize}
\item If $\beta$ is in the domain of $\eta_{i}$, then it is also in the domain of $\eta_{j}$ and
$$B[\beta/x]\eta_{i}\sigma_{i}=B\eta_{i}\sigma_{i}[\eta_{i}(\beta)/x]$$
$$C[\delta/x]\eta_{j}\sigma_{j}=C\eta_{j}\sigma_{j}[\eta_{j}(\delta)/x]$$
By Definition \ref{def:exp_preproof} of weak regularity, we suppose $\eta_{i}(\beta)=\eta_{j}(\delta)$ and then we have to check that: i)  $S$ and $R$ are both trees or cuts; ii)    $$s \eta_{i}\sigma_{i}, \forall x B\,\eta_{i}\sigma_{i}=r \eta_{j}\sigma_{j}, \forall x C\,\eta_{j}\sigma_{j}$$
First, we have $i=j$ and $\beta=\delta$, because $\eta_{i}$ and $\eta_{j}$ are both injective and have disjoint domains by the freshness assumption on the renamings. 
 Now, since either $\overline{s}\in \mathsf{B}$ and $\overline{r}\in \mathsf{B}$ or $J_{1}x A, \overline{s}\in \mathsf{B}$ and $J_{2}x A,\overline{r}\in \mathsf{B}$,  by weak regularity of $\calP_{1}$ we infer i) and $$s, \forall x B= r, \forall x C$$ thus ii) also follows.
 
 \item If $\beta$ is not in the domain of $\eta_{i}$, then it is not in the domain of $\eta_{j}$ either and 
$$B[\beta/x]\eta_{i}\sigma_{i}=B\eta_{i}\sigma_{i}[\beta/x]$$
$$C[\delta/x]\eta_{j}\sigma_{j}=C\eta_{j}\sigma_{j}[\delta/x]$$
since the $\forall$-expansions $+^{\alpha_{i}}, +^{\alpha_{j}}$ by assumption do not occur in $\calP$ and thus in $\calQ$. By Definition \ref{def:exp_preproof} of weak regularity, we suppose $\beta=\delta$ and then we have to check that: i)  $S$ and $R$ are both trees or cuts; ii)    $$s \eta_{i}\sigma_{i}, \forall x B\,\eta_{i}\sigma_{i}=r \eta_{j}\sigma_{j}, \forall x C\,\eta_{j}\sigma_{j}$$
Now, since either $\overline{s}\in \mathsf{B}$ and $\overline{r}\in \mathsf{B}$ or $J_{1}x A, \overline{s}\in \mathsf{B}$ and $J_{2}x A,\overline{r}\in \mathsf{B}$,  by weak regularity of $\calP_{1}$ we infer i) and $$s, \forall x B= r, \forall x C$$

If $i=j$, we are done. So assume $i\neq j$. We want to show that  no variable in the domain of $\eta_{i} \sigma_{i}$ or of  $\eta_{j} \sigma_{j}$ appears in  $s, \forall x B$ or $s, \forall x C$, so that we are done. Indeed, if there was such a variable $\gamma$, then there would be either some expansion $v$ containing $\gamma$ or some cut $v$ with $\Shallow(v)$  containing $\gamma$ such that  ${v}$ dominates $+^{\beta}$ in $\calP$ and thus ${v}<_{\calP_{1}}+^{\beta}$.  By choice of $\eta_{i}, \eta_{j}$, we have  $+^{\alpha_{k}}<_{\calP_{1}}+^{\gamma}$, where $k\in\{i, j\}$ and moreover $+^{\gamma}<_{\calP_{1}} {v}$; thus putting the three together,  $+^{\alpha_{k}}<_{\calP_{1}}+^{\gamma}<_{\calP_{1}} {v}<_{\calP_{1}}+^{\beta}$. Therefore, $\beta$ is in the domain of $\eta_{k}$, though we were assuming it is not. 
%We first show that$$\calQ= \calP\eta_1\sigma_{1}\omerge\ldots \omerge \calP\eta_n\sigma_{n}$$ is weakly regular.
\end{itemize}

\textbf{Validity}. The formula
$$\Deep(\calP_1)=\big(\bigvee_{i=1}^n\Deep(E_i)\land \Deep(F_{1})\big)\lor \ldots \lor \big(\bigvee_{i=p}^l\Deep(E_i)\land \Deep(F_{q})\big)  \lor\Deep(\calP)$$
by assumption is valid and  logically implies the formulas
$$(1)\ \bigvee_{i=1}^n\Deep(E_i)\lor\ldots \lor \bigvee_{i=p}^l\Deep(E_i)\lor \Deep(\calP)$$
$$\Deep(F_{1})\lor\ldots \lor \Deep(F_{q})\lor \Deep(\calP)$$
which then are valid too. Therefore, also the formula

$$(2)\ \bigwedge_{i=1}^{l}\big(\Deep(F_{1})\eta_i\sigma_i\lor \ldots \lor \Deep(F_{q})\eta_i\sigma_i\lor \Deep(\calP)\eta_i\sigma_i\big)$$
is valid. We have that $\Deep(\calQ)$ is equal to the formula
$$(\Deep(E_{1})\land \Deep(F_{1})\eta_1\sigma_1) \lor \ldots \lor (\Deep(E_{1})\land \Deep(F_{q})\eta_1\sigma_1)\lor $$$$ \ldots \lor (\Deep(E_{l})\land \Deep(F_{1})\eta_l\sigma_l) \lor \ldots \lor (\Deep(E_{l})\land \Deep(F_{q})\eta_l\sigma_l)\lor$$$$ \Deep(\calP)\lor\bigvee_{i=1}^l\Deep(\calP\eta_i\sigma_i)$$
%$$\Deep(\calQ)=\bigvee_{i=1}^{l}\big(\Deep(E_{i})\land (\Deep(E)\eta_i\sigma_i \lor \ldots \lor \Deep(F)\eta_i\sigma_i)\big)\lor \Deep(\calP)\lor\bigvee_{i=1}^l\Deep(\calP\eta_i\sigma_i)$$
Now, fix any propositional truth assignment. Assume $\Deep(\calP)\lor\bigvee_{i=1}^l\Deep(\calP\eta_i\sigma_i)$ is false under this assignment. By (1), we can assume $\Deep(E_{1})$ is true (if one among $\Deep(E_{2}), \ldots, \Deep(E_{l})$ is true, the reasoning is symmetric). %We want to show that the formula $\bigvee_{i=1}^{n}\big(\Deep(E_{i}\land\Deep(E)\eta_i\sigma_i\big)$ is true, so that $\Deep(\calQ)$ is true. 
Indeed, by (2), the formula $\Deep(F_{1})\eta_1\sigma_1\lor \ldots \lor \Deep(F_{q})\eta_1\sigma_1$ is true, thus the formula $$(\Deep(E_{1})\land \Deep(F_{1})\eta_1\sigma_1) \lor \ldots \lor (\Deep(E_{1})\land \Deep(F_{q})\eta_1\sigma_1)$$ is true and finally $\calQ$ is true.

 %Using propositional reasoning, in particular validity of $(A\lor B)\land (C\lor D) \impl A\lor (B \land C)\lor D$, we obtain validity of $F'$.

\textbf{Acyclicity}. We show that acyclicity of $<_{\calP_1}$ implies acyclicity of $<_{\calQ}$. Our strategy is to map any purported cycle of $\calQ$ into a cycle of $\calP_{1}$. 

\noindent The first difficulty we are going to face is that the expansions of $\calQ$ are not well-behaved copies of expansions of $\calP_{1}$, because of the substitutions $\sigma_{i}$. For example, an $\exists$-expansion $+^{u}$ of $\calP$ could contain some $\alpha_{j}$, so that we would find $+^{u[t_{i}/\alpha_{j}]}$ in $\calQ$. The trouble is that $+^{u[t_{i}/\alpha_{j}]}$ could contain variables that were not in $u$, thus there could be $\forall$-expansions  $+^{\beta}$ such that $+^{\beta}<_{\calQ} +^{u[t_{i}/\alpha]}$ but  it is not the case that $+^{\beta}<_{\calP_{1}} +^{u}$. This means that the relation $<_{\calQ}$ presents entirely new paths that were not in $\calP_{1}$. We might encounter paths in $\calQ$ that a priori could not be mapped back to paths in $\calP_{1}$: for example, a path featuring copies of old expansions suddenly followed by new ill-behaved expansion like $+^{u[t_{i}/\alpha_{j}]}$. However, we shall take care of this issue in a next  Lemma, explaining that such paths cannot end up in a cycle, because whatever path enters the ``renamed zone'' is stuck in it. The basic intuition is that since $u$ contains $\alpha_{j}$, everything in relation $<_{\calQ}$ with $+^{u[t_{i}/\alpha_{j}]}$ is under the scope of the renaming $\eta_{j}$ and all the renamed $\forall$-expansion have a renamed variable that can ``jump'' only in the renamed zone.

\noindent The second difficulty we are going to encounter is that one of the new displayed cuts $\{E_{k}, F_{h}\eta_{k}\sigma_{k}\}$ of $\calQ$ may have no direct correspondents in $\calP_{1}$. Again, the problem is that the term $t_{k}$ introduced by the substitution $\sigma_{k}$ could add some variable $\beta$ to the shallow formula $A[t_{k}/x]$ of the new cut which is not in the shallow formula $\exists x A$ of the original cuts. In this case, some $\forall$-expansion $+^{\beta}$ would be in relation with the new cut $\{E_{k}, F_{h}\eta_{k}\sigma_{k}\}$ of $\calQ$, but with none of the old ones of $\calP_{1}$. In this case the new cut will be mapped to $+^{t_{k}}$, in all others to the old cut $\{  \exists x\, A +^{t_i} E_i \cdots +^{t_j} E_j, \forall x\, \bar{A} +^{\alpha_{s}} F_{s} \}$ such that $i\leq k\leq j$. But when the new cut is mapped to $+^{t_{k}}$, one has to make sure that the new cut is not in relation in the cycle with some expansion in $F_{h}\eta_{k}\sigma_{k}$, otherwise the mapping we wish to build would fail.  Luckily, the \emph{old} expansion $+^{\beta}$ of $\calP_{1}$ ($\beta$ occurs in $t_{k}$) cannot jump in the ``renamed zone'' if it is part of a cycle. The ``renamed zone'' argument settles the issue also when $s\neq h$, which would cause our mapping to fail in the way we have just discussed.

Let us now work out the formal details. Consider any cycle
$$v_{1}<_{\calQ} v_{2} <_{\calQ}\ldots <_{\calQ} v_{m} <_{\calQ} v_{1}$$
 in $\calQ$. Then each $v_{i}$ is either  of the form $w_{i}\eta_{k}\sigma_{k}$ for some  $w_{i}$ which  is an old expansion or cut belonging to $\calP_{1}$ or $v_{i}$ is one of the new displayed  cuts $\{E_{k}, F_{h}\eta_{k}\sigma_{k}\}$ of $\calQ$; in this second case, we define $w_{i}$ to be the  displayed cut $\{  \exists x\, A +^{t_i} E_i \cdots +^{t_j} E_j, \forall x\, \bar{A} +^{\alpha_{s}} F_{s} \}$ of $\calP_{1}$ such that $i\leq k\leq j$, if the eigenvariable of the $\forall$-expansion $v_{i-1 (mod\ m)}$ does not occur in $t_{k}$, or to be the displayed occurrence of $+^{t_{k}}$ in $\calP_{1}$ otherwise. We want to show that 
$$w_{1}<_{\calP_1} w_{2} <_{\calP_1}\ldots <_{\calP_1} w_{m} <_{\calP_1} w_{1}$$
First of all we need the following:\\

\emph{``Renamed Zone'' Lemma}. Suppose $\alpha\in\{\alpha_{1}, \dots, \alpha_{q}\}$. If there are  $k$ and $j>0$ such that $+^\alpha<_{\calP_{1}}w_{k}$ and $v_{k}$ occurs in $\calP\eta_{j}\sigma_{j}$ or $F_{h}\eta_{j}\sigma_{j}$, then for all $i$, $+^{\alpha}<_{\calP_{1}}w_{i}$ and  $v_{i}$ occurs in $\calP\eta_{j}\sigma_{j}$ or $F_{h}\eta_{j}\sigma_{j}$.\\

\emph{Proof of the Lemma}. Since we are dealing with a cycle, we may assume without loss of generality that $k=1$. We proceed by induction on $i$, the case $i=1$ being already settled. Suppose by induction  hypothesis that $+^{\alpha}<_{\calP_{1}}w_{i}$ and  $v_{i}$ occurs in $\calP\eta_{j}\sigma_{j}$ or $F_{h}\eta_{j}\sigma_{j}$ and thus $w_{i}$ occurs in $\calP$ or $F_{h}$. If $v_{i}$ dominates $v_{i+1}$, then $w_{i}$ dominates $w_{i+1}$ 
 and thus $+^{\alpha}<_{\calP_1}w_{i+1}$ and $v_{i+1}$ occurs in $\calP\eta_{j}\sigma_{j}$ or $F_{h}\eta_{j}\sigma_{j}$. Suppose then $v_{i}$ is a $\forall$-expansion, thus also $w_{i}=+^\beta$ is a $\forall$-expansion and the eigenvariable of $v_{i}$ occurs in $v_{i+1}$ or $\Shallow(v_{i+1})$. Since $+^\alpha<_{\calP_{1}}+^\beta$ and $v_{i}\in \calP\eta_{j}\sigma_{j}$ or $F_{h}\eta_{j}\sigma_{j}$, we have $v_{i}=+^{\eta_{j}(\beta)}$. Now, $\eta_{j}(\beta)$ is fresh and as it occurs in $v_{i+1}$ or $\Shallow(v_{i+1})$, we have that $v_{i+1}$ must occur in $ \calP\eta_{j}\sigma_{j}$ or $F_{h}\eta_{j}\sigma_{j}$. Moreover, $v_{i+1}$ cannot be one of the new cuts of $\mathcal{Q}$, because $\Shallow(w)$ would not contain $\eta_{j}(\beta)$. Therefore, $v_{i+1}=w_{i+1}\eta_{j}\sigma_{j}$; since $\eta_{j}(\beta)$ occurs in $v_{i+1}$ and is fresh, it also occurs in $w_{i+1}\eta_{j}$ and thus $\beta$ must occur in $w_{i+1}$ or $\Shallow(w_{i+1})$, so that $+^{\alpha}<_{\calP_{1}}+^{\beta} <_{\calP_{1}} w_{i+1}$, which ends the proof of the Lemma.\\

We now prove, by induction on $i$, that for every $i$, $w_{i}<_{\calP_1}w_{i+1}$ (in the following the indexes $i$ of $w_{i}$ and $v_{i}$ will be taken modulo $m$).
 
  If $v_{i}$ dominates $v_{i+1}$, there are three possibilities:
  \begin{itemize}
  \item $w_{i}$ is an old cut or expansion of $\calP_{1}$ and $v_{i}=w_{i}\eta_{k}\sigma_{k}$. Then also $v_{i+1}=w_{i+1}\eta_{k}\sigma_{k}$, where $w_{i+1}$ is an old expansion of $\calP_{1}$,
  therefore also $w_{i}$ dominates $w_{i+1}$ and we get $w_{i}<_{\calP_1}w_{i+1}$. %and $+^{\alpha}<_{\calP_1}w_{i+1}$. 

  \item $w_{i}$ is the  displayed cut $\{  \exists x\, A +^{t_i} E_i \cdots +^{t_j} E_j, \forall x\, \bar{A} +^{\alpha_{s}} F_{s} \}$ of $\calP_{1}$ such that $i\leq k\leq j$ and $v_{i}=\{E_{k}, F_{h}\eta_{k}\sigma_{k}\}$. % and  $v_{i-1}=+^{\beta}$ must be a $\forall$-expansion, with $\beta$ not occurring in $t_{k}$.  
  Since $v_{i}$ dominates $v_{i+1}$, either $v_{i+1}$ occurs in $E_{k}$ or $v_{i+1}$ occurs in $F_{h}\eta_{k}\sigma_{k}$. In the first case,  $v_{i+1}=w_{i+1}$, thus $w_{i}<_{\calP_1}w_{i+1}$. The second case  is not possible: if $v_{i+1}$ occurred in $F_{h}\eta_{k}\sigma_{k}$, then $v_{i+1}=w_{i+1}\eta_{k}\sigma_{k}$ with $w_{i+1}$ occurring in $F_{h}$. Therefore, $+^{\alpha_{h}}<_{\calP_{1}}w_{i+1}$ and by the ``Renamed Zone'' Lemma, $+^{\alpha_{h} }<_{\calP_{1}}w_{i-1}$ and $v_{i-1}$ occurs in $\calP\eta_{k}\sigma_{k}$ or $F_{h}\eta_{k}\sigma_{k}$.  Moreover,  $v_{i-1}=+^{\beta}$, with $\beta$ occurring in $\overline{A}[t_{k}/x]$. %, but not in $t_{k}$. 
  Thus $\beta$ must be a variable of $\calP_{1}$ and $+^{\alpha_{h}}<_{\calP_{1}}w_{i-1}=+^{\beta}$. But then $+^{\beta}$ cannot occur  in $\calP\eta_{k}\sigma_{k}$ or $F_{h}\eta_{k}\sigma_{k}$, because $\beta$ is in the domain, but not in the range, of $\eta_{k}$: contradiction.

    \item $w_{i}=+^{t_{k}}$, $v_{i}=\{E_{k}, F_{h}\eta_{k}\sigma_{k}\}$. %and $v_{i+1}$ occurs in $E_{k}$. 
    Since $v_{i}$ dominates $v_{i+1}$, either $v_{i+1}$ occurs in $E_{k}$ or $v_{i+1}$ occurs in $F_{h}\eta_{k}\sigma_{k}$. The second case is excluded as before. In the first case,  $v_{i+1}=w_{i+1}$,   therefore $w_{i}<_{\calP_1}w_{i+1}$.
    \end{itemize}
  Suppose then $v_{i}$ is an $\forall$-expansion and thus $w_{i}=+^\gamma$ is a $\forall$-expansion as well. %Since $+^{\alpha}<_{\calP_1}+^{\gamma}$, by definition $\gamma$ is in the domain of $\eta_{k}$ for all $1\leq k\leq m$. 
 We have two cases.
  
 \begin{enumerate}
 \item  $v_{i}=+^{\eta_{j}(\gamma)}$. Now, as $v_{i}<_{\calQ}v_{i+1}$, we know that $\eta_{j}(\gamma)$  occurs in $v_{i+1}$ or in $\Shallow(v_{i+1})$.  Moreover, $v_{i+1}$ cannot be one of the new displayed cuts of $\calQ$, because $\Shallow(v_{i+1})=A[t_{k}/x]$ or $\Shallow(v_{i+1})=\overline{A}[t_{k}/x]$ and since $\eta_{j}(\gamma)$   is fresh, it cannot occur in those formulas.  
 
Thus $v_{i+1}$ is the result of a substitution in an old cut (resp. expansion) $w_{i+1}$,   so $\Shallow(v_{i+1})=\Shallow(w_{i+1})\eta_{k}[t_{k}/\alpha]$ (resp. $v_{i+1}=w_{i+1}\eta_{k}[t_{k}/\alpha]$). Since $\eta_{j}(\gamma)$   is fresh, it cannot occur in $t_{k}$, therefore $j=k$ and $\gamma$ must occur also in $\Shallow(w_{i+1})$ (resp. $w_{i+1}$) and thus %$+^{\alpha}<_{\calP_1}+^{\gamma}<_{\calP_1}w_{i+1}$ and 
 $w_{i}=+^\gamma <_{\calP_1}w_{i+1}$.

\item $v_{i}=+^\gamma$ and $\gamma$ occurs in $v_{i+1}$ or in $\Shallow(v_{i+1})$.  Now we are left with two possibilities.
 
 \begin{itemize}
 \item  $v_{i+1}=w_{i+1}\eta_{k}\sigma_{k}$. If $k=0$, then $v_{i+1}=w_{i+1}$ and we are done. Moreover, if no $\alpha\in\{\alpha_{1}, \dots, \alpha_{q}\}$  occurs in $w_{i+1}$ or $\gamma$ does not occur in $t_{k}$, then $\gamma$ occurs in $w_{i+1}$, which means $w_{i}<_{\calP_{1}}w_{i+1}$. Suppose thus by contradiction that they do. Then $+^{\alpha}<_{\calP_{1}}w_{i+1}$ and $v_{i+1}$ occurs in $\calP\eta_{k}\sigma_{k}$ or $F_{h}\eta_{k}\sigma_{k}$. By the ``Renamed Zone'' Lemma, we conclude that $+^\alpha<_{\calP_{1}}+^\gamma$ and $+^\gamma$ occurs in $\calP\eta_{k}\sigma_{k}$ or $F_{h}\eta_{k}\sigma_{k}$: but then $\gamma$ is in the domain of  $\eta_{k}$, whereas $+^{\gamma}$ occurs in $\calP$ or $F_{h}$, contradiction.\\

\item 
 $v_{i+1}=\{E_{k}, F_{h}\eta_{k}\sigma_{k}\}$ is one of the new cuts of $\mathcal{Q}$.  Since $v_{i}<_{\calQ} v_{i+1}$,  $\gamma$ occurs in $\Shallow(v_{i+1})=A[t_{k}/x]$.  If $\gamma$ does not occur in $t_{k}$, then by definition of $w_{i+1}$,  $\gamma$ occurs in $$\exists x A=\Shallow(\{  \exists x\, A +^{t_i} E_i \cdots +^{t_j} E_j, \forall x\, \bar{A} +^\alpha F_{s} \})=\Shallow(w_{i+1})$$ with $i\leq k\leq j$, so $w_{i}<_{\calP}w_{i+1}$. If $\gamma$ does occur in $t_{k}$, then $w_{i+1}=+^{t_{k}}$, so $w_{i}<_{\calQ}w_{i+1}$.

 %Suppose then $\gamma$  occurs in $t_{k}$ and thus  $w_{i+1}=+^{t_{k}}$. We show that $v_{i+2}$ must occur in $E_{k}$, so that $+^{\gamma}<_{\calP_{1}} w_{i+1}$, as desired. Indeed, suppose by contradiction that $v_{i+2 }$ does not occur in $E_{k}$ and thus occurs in $E\eta_{k}\sigma_{k}$. 
%Since $v_{i+2 }$ occurs in $E\eta_{k}\sigma_{k}$, it follows that $w_{i+2}$ is dominated by $+^{\alpha}$ in $\calP_{1}$, thus $+^{\alpha}<_{\calP_{1}}w_{i+2 }$. By the ``Renamed Zone'' Lemma, we conclude that $+^\alpha<_{\calP_{1}}\gamma$ and $\gamma$ occurs in $\calP\eta_{j}\sigma_{j}$ or $E\eta_{j}\sigma_{j}$, so that it must be the case that $v_{i}=+^{\eta_{j}(\gamma)}\neq +^\gamma$, contradiction.\\

\end{itemize}

\end{enumerate}

\textbf{Eigenvariable condition.} The fact that the eigenvariable of every $\forall$-expansion of $\calQ$ does not occur in $\Shallow(\calQ)$ is ensured by $\Shallow(\calQ)=\Shallow(\calP)$ and the new $\forall$-expansion having fresh eigenvariables.

\end{proof}
%

%
%
%
%For quantifier nodes
%$n$, let $\bvar(n)$ be the variable bound in the quantifier node.
%\\
%TODO: For my arguments, I would also like to use the dependency relation on $\forall$-nodes
%or equivalently, their eigenvariables. For this purpose, we could simply define a new
%order $<'$ on expansion terms, cuts, and eigenvariables by postulating $x<'y$ if either
%$x<y$ or $x$ an expansion term and $y$ an eigenvariable and and $t$ dominates $q(y)$ or
%$x$ an eigenvariable and $y$ an expansion term and $x\in\Var(y)$, or $x$ is a cut
%and $y$ is an eigenvariable such that $q(y)$ is in $x$, or $x$ is an eigenvariable
%and $y$ is a cut and $x\in\Var(C)$ where $C$ is the cut-formula of $y$ or $x$ is an eigenvariable
%and $y$ is an eigenvariable and $q(x)$ dominates $q(y)$.
%We will pretend for
%now that our relation $<$ includes this. One should show that $<'$ restricted
%to expansion terms and cuts is exactly $<$, and that $<$ acyclic implies $<'$ acyclic.
%

\subsection{Complexity Measures}
Let $\rightarrow$ denote the reflexive, transitive closure of the mapping $\mapsto$. Our next aim is to prove weak normalization of our reduction
system $\rightarrow$. It turns out that a parallel version of the proof strategy for cut-elimination can be applied to expansion trees.
Equipped with this observation, we can adapt to our setting the notion
of {\em rank} of a cut and the notion of \emph{maximal} cut of maximal rank, which in turn will allow us to prove
weak normalization. In fact, these notions can be formulated in a natural way using the
language of expansion trees we have introduced so far. In sequent calculus, a maximal cut  is just a cut of maximal rank having no cut of the same rank above it in the proof tree. In our setting, the geometry of the proof is represented by the dependency relation between cuts, so a maximal cut of  maximal rank is just a cut of maximal rank which is not smaller, according to the dependency relation, than any cut of that rank.
\begin{definition}[Rank of a Cut, Maximal Cut]
Let $\calP$ be an expansion proof with merges and $C$ a cut of $\calP$. We define $\rk(C)$ as the logical complexity of $\Shallow(C)$ and we call it the \emph{rank} of $C$. We call $C$  \emph{maximal} if for all cuts $D$ of $\calP$, $\rk(D)\leq \rk(C)$ and $\rk(D)=\rk(C)$ implies that $C\nless_{\calP} D$.
\end{definition}

\subsection{Weak Normalization}\label{sec:weak_normalization}

This section is dedicated to proving that there exists a terminating
strategy for the application of the cut-reduction rules.
Given an expansion proof $\calP$, our reduction strategy will be based 
on picking maximal cuts and reducing them in parallel. 

\begin{theorem}[Weak Normalization]\label{thm:weak_normalization}
For every expansion proof  $\calQ$ there is a cut-free expansion proof $\calQ^*$
such that $\calQ \rightarrow \calQ^*$ and $\Shallow(\calQ) = \Shallow(\calQ^*)$.
\end{theorem}

\begin{proof}
We partition the collection of cuts occurring in $\calQ$ in equivalence classes, by means of the equivalence relation
$$C_{1}\sim C_{2}$$
$$\mbox{ iff } $$
$$C_{1}=\{  \exists x\, A +^{t_1} E_1 \cdots +^{t_n} E_n , \forall x\, \bar{A} +^\alpha F \}\mbox{ and } C_{2}= \{  \exists x\, A +^{s_1} G_1 \cdots +^{s_m} G_m , \forall x\, \bar{A} +^\beta H \}$$

We now proceed by induction on the pair $(r, k)$, where $r$ is the greatest among the ranks of the cuts in $\calQ$ and $k$ is the number of equivalence classes whose cuts have rank $r$. If $\calQ$ is already cut-free, we are done. Otherwise, we wish to single out a \emph{maximal} equivalence class: an equivalence  class whose cuts are all maximal. 

We first prove that a maximal equivalence class exists. Consider the relation $\prec$ between equivalence classes defined as follows: $\mathcal{A}\prec\mathcal{B}$ if and only if there is a $C\in \mathcal{A}$ such that $C<_{\calQ} D$ for every $D$ in $\mathcal{B}$. 
\noindent We begin by showing that this relation is not cyclic. Suppose indeed by contradiction that $\mathcal{A}_{1}\prec\ldots \prec \mathcal{A}_{n}\prec \mathcal{A}_{1}$. For every $i$, let $C_{i}\in \mathcal{A}_{i}$ be a cut such that  $C_{i}<_{\calQ} D$ for every $D$ in $\mathcal{A}_{i+1}$ or for every $D$ in $\mathcal{A}_{1}$, if $i=n$. By construction, $C_{1}<_{\calQ}\ldots  <_{\calQ} C_{n}<_{\calQ} C_{1}$, a contradiction, because $<_{\calQ}$ is acyclic. 

\noindent Secondly, we show that that for any  equivalence classes $\mathcal{A}$ and $\mathcal{B}$, if we assume there are $C\in\mathcal{A}$,  $D\in \mathcal{B}$ such that  $C<_{\calQ} D$, then  $\mathcal{A}\prec \mathcal{B}$. Indeed, since $C<_{\calQ} D$, we have a chain of cuts or expansions $C<_{\calQ} w_{1}<_{\calQ}\ldots  <_{\calQ} w_{n}<_{\calQ} D$; moreover, by Definition \ref{defi-dependency}, $w_{n}$ must be a $\forall$-expansion $+^{\beta}$ such that $\beta$ occurs in $\Shallow(D)$. By definition of the relation $\sim$, for every $E\in \mathcal{B}$, we have $\Shallow(E)=\Shallow(D)$. Therefore, $C<_{\calQ} w_{1}<_{\calQ}\ldots  <_{\calQ} w_{n}<_{\calQ} E$. We conclude $\mathcal{A}\prec \mathcal{B}$.

\noindent Third, suppose by contradiction that there is no maximal equivalence class, assuming there is at least one equivalence class. We want to show that for every equivalence class $\mathcal{A}$ whose cuts have rank $r$, there is an equivalence class $\mathcal{B}$ whose cuts have rank $r$ such that $\mathcal{A}\prec \mathcal{B}$; since there are finitely many equivalence classes, this implies that the relation $\prec$ is cyclic, a contradiction. Indeed, consider any equivalence class $\mathcal{A}$ whose cuts have rank $r$. By assumption, there is a cut $C\in \mathcal{A}$ such that $C<_{\calQ} D$, with $D$ of rank $r$. Let $\mathcal{B}$ the equivalence class of $D$. By definition of $\sim$, all cuts of $\mathcal{B}$ have rank $r$, and by what we have proved above, $\mathcal{A}\prec \mathcal{B}$, which is what we wanted to show.

Now let us consider the possible reduction steps.

\textbf{Quantifier Step}. If there is at least a cut whose shallow formula is an existential formula,  let $\sigma_1, \ldots, \sigma_{l}$ be respectively the substitutions $[{t_1}/\alpha_{1}\dots t_{1}/\alpha_{q}], \ldots, [{t_l}/\alpha_{1}\dots t_{l}/\alpha_{q}]$, $\mathcal{A}$ be a maximal equivalence class and 
$$\mathcal{A}=\{  \exists x\, A +^{t_1} E_1 \cdots +^{t_n} E_n , \forall x\, \bar{A} +^{\alpha_{1}} F_{1} \}, \ldots, \{  \exists x\, A +^{t_{p}} E_p \cdots +^{t_{l}} E_l , \forall x\, \bar{A} +^{\alpha_{q}} F_{q} \}$$
$$\calQ = \mathcal{A}, \calP$$
$$\calQ'= \{ E_1, F_{1}\eta_1\sigma_{1}\},\ldots, \{ E_1, F_{q}\eta_1\sigma_{1}\},\ldots, \{E_{l}, F_{1}\eta_l\sigma_{l} \},\ldots,   \{E_{l}, F_{q} \eta_l\sigma_{l} \}, \calP, \calP\eta_1\sigma_{1}, \ldots, \calP\eta_l  \sigma_{l}$$
where the $\forall$-expansions $+^{\alpha_{1}}, \dots, +^{\alpha_{q}}$ do not occur in $\calP$, no cut in $\calP$ has shallow formula $\exists x A $ and $\eta_{1}, \dots, \eta_{l}$ are renamings to fresh variables of the eigenvariables $\beta$ of $\calP, F_{1},\ldots, F_{q}$ such that for some $1\leq i\leq q$ and occurrence of $+^{\alpha_{i}}$ we have    $+^{\alpha_{i}}<_{\calP_{1}}+^{\beta}$. First of all, we observe that it is alway possible to satisfy the condition that no  $+^{\alpha_{i}}$ occurs in $\calP$:  by weak regularity of $\calQ$, every occurrence of $+^{\alpha_{i}}$ is on the right of the shallow formula $\forall x \overline{A}$ of some cut.

Let 
$$D=\{  \exists x\, B +^{s_1} G_1 \cdots +^{s_m} G_m , \forall x\, \bar{B} +^\beta H \}$$ be any cut of  rank $r$ in $\calP$. We want to prove that neither $\alpha_{i}$ nor any variable $\gamma$ in the domain of any $\eta_{k}$ can occur in $\exists x\, B $ or be equal to $\beta$. Indeed, $\alpha_{i}$ by hypothesis must be different from $\beta$ and if it occurred in $\exists x\, B $,  we would  have by Definition \ref{defi-dependency} that for every $C\in\mathcal{A}$, $C<_{\calQ}+^{\alpha}<_{\calQ}D$, contradicting the maximality of $C$. Moreover, if  $\gamma$ in the domain of any $\eta_{k}$ occurred in $\exists x\, B $, then for some $C\in\mathcal{A}$, $C<_{\calQ}+^{\alpha_{i}}<_{\calQ}+^{\gamma}<_{\calQ}D$, contradicting the maximality of $C$; on the other hand, if $\gamma$ were equal to $\beta$, then  for some $C\in\mathcal{A}$, $C<_{\calQ}+^{\alpha_{i}}<_{\calQ}+^{\beta}$, therefore we would have a chain of expansions or cuts $+^{\alpha_{i}}<_{\calQ}w_{1}<_{\calQ}\ldots  <_{\calQ} w_{n}<_{\calQ}+^{\beta}$, so that $w_{n}=D$, contradicting the maximality of $C$.

 Let now $D_{1}, \ldots, D_{m}$ be the cuts of rank $r$ in $\calP$. Then, for each $i$, 
$$\calP\eta_i\sigma_{i}= D_{1}\eta_i\sigma_{i}, \ldots, D_{m}\eta_i\sigma_{i}, \calP_{i}$$
$$\calP=D_{1}, \ldots, D_{m}, \mathcal{P}_{0}$$
with no cut of rank $r$ appearing in $\calP_{i}$. Let 
$$\mathcal{D}=D_{1}, \ldots, D_{m}, D_{1}\eta_1\sigma_{1}, \ldots, D_{m}\eta_1\sigma_{1}, \ldots, D_{1}\eta_l\sigma_{l}, \ldots, D_{m}\eta_l\sigma_{l} $$
Then
$$\calQ'= \{ E_1, F_{1}\eta_1\sigma_{1}\},\ldots, \{ E_1, F_{q}\eta_1\sigma_{1}\},\ldots, \{E_{l}, F_{1}\eta_l\sigma_{l} \},\ldots,   \{E_{l}, F_{q} \eta_l\sigma_{l} \}, \mathcal{D},\calP_{0},\ldots, \calP_{n}$$
By what we have proved, for every $i, j$, if 
$$D_{i}=\{  \exists x\, B +^{s_1} G_1 \cdots +^{s_m} G_m , \forall x\, \bar{B} +^\beta H \}$$
then
$$D_{i}\eta_j\sigma_{j}=\{  \exists x\, B +^{s_1'} G_1' \cdots +^{s_m'} G_m', \forall x\, \bar{B} +^\beta H' \}$$
Therefore, the number of equivalence classes of rank $r$ in $\calQ'$ is the number of classes in which $\mathcal{D}$ is partitioned, that is exactly the number of classes in which $D_{1}, \ldots, D_{m}$ is partitioned: $k-1$. 
 By induction hypothesis, $\calQ'\rightarrow \calQ^*$, with $Q^{*}$ cut-free, which is the thesis.

\textbf{Propositional Step}.  Assume
%$$C=\{ E_1 \lor F_1, E'_1\land F'_1 \}$$
$$\calQ= \{ E_1 \lor F_1, E'_1\land F'_1 \},\dots, \{ E_m \lor F_m, E'_m\land F'_m \},  \calP$$
$$\calQ'=\{ E_1, E'_1 \}, \{ F_1, F'_1 \},\dots, \{ E_m, E'_m \}, \{ F_m, F'_m \}, \calP$$
where $\calP$ does not contain any other cut of rank $r$. 
Then, the number of equivalence classes of cuts of rank $r$ in $\calQ'$ is strictly smaller than $k$, because propositional cuts are not in relation $\sim$ with any other cut and form singleton classes. By induction hypothesis, $\calQ'\rightarrow \calQ^*$, with $Q^{*}$ cut-free, which is the thesis.

\textbf{Atomic Step}. As in the previous case.

\end{proof}

\subsection{Strong Normalization}\label{sec:strong_normalization}
%TODO: if possible write that we conjecture strong SN, try with non-SN examples
%of Willem and Richard, say why they fail
%
Having shown weak normalization of the cut-reduction rules in the previous
section, it is important to turn to the question of strong normalization,
i.e.~whether {\em all} reduction sequences are of finite length. We conjecture
that our cut-reduction rules are indeed strongly normalizing, and present some
evidence for this claim by discussing how our reduction rules behave on
a translation of the example~\cite[Figure 16]{Heijltjes10Classical}, which shows how bridges can cause infinite loops in the setting of proof forests.

This example can be translated as an expansion proof
of the form $(C^+,C^-),\top$  with $\top$ the atomic formula ``true'' and
\[
\begin{array}{lll}
C^+=&\exists x\, \dual{P(x)} &+^c\,  \dual{P(c)} +^{f(\alpha)} \dual{P(f(\alpha))}\\
C^-=&\forall x\, P(x) &+^\alpha\, P(\alpha)\\
\end{array}
\]
We have no issue at all:
\begin{footnotesize}$$(C^+,C^-)\mapsto \{P(c), \dual{P(c)}\}, \{P(f(\alpha)), \dual{P(f(\alpha))}\}, \top, \top[c/\alpha], \top[f(\alpha)/\alpha]\mergered \{P(c), \dual{P(c)}\}, \{P(f(\alpha)), \dual{P(f(\alpha))}\}, \top\mapsto \top$$\end{footnotesize}
 This is essentially due to the different treatment of bridges (i.e.~dependencies
between different sides of a cut, see Section~\ref{sec:cut_reduction_steps}) in our
formalism: at the core of the non-termination of~\cite[Figure 14]{Heijltjes10Classical}
lies the bridge in $(C^{+}, C^{-})$~\cite[Figure 16]{Heijltjes10Classical} which induces a cycle.
In the setting of proof forests, the non-termination due to bridges is handled by adding a pruning reduction,
having the task of removing bridges as soon as they appear.
%The trouble is that this operation feels as a technical patch and it is not clear why one should hope that strong normalization holds.
In our setting we are able to get by naturally without pruning. This is due to our different renaming and duplication policy: not everything greater in the
dependency relation than the cut is duplicated and renamed. In particular, the expansion $+^{f(\alpha)}$ is not duplicated, even if it is above $\alpha$ in the dependency relation.
%%FEDERICO: I don't think that's the reason
%One explanation for the fact that in our setting, we are able to get by without such a reduction, is the use of the merge in the definition of the cut-reduction rules. The merge has the advantage that it s very natural, it is an extension of the merge for cut-free expansion  proofs from~\cite{Miller87Compact}, and it is useful also in applications not related to cut-elimination, as in the proof of Theorem~\ref{thm:completeness}.
%
\subsection{Confluence}
It is well-known that cut-elimination and similar procedures in classical logic are typically
non confluent, see e.g.~\cite{Urban00Classical,Ratiu12Exploring,Baaz05Experiments}
for case studies and~\cite{Baaz11Nonconfluence,Hetzl12Computational} for asymptotic results.
Neither the proof forests of~\cite{Heijltjes10Classical} nor the Herbrand nets 
of~\cite{McKinley13Proof} have a confluent reduction. The situation is analogous
in our formalism: the reduction is not confluent. In fact, one can use the same example
to demonstrate this; let
\begin{align*}
\calP = & \{ \exists x\, A +^s A\unsubst{x}{s} +^t A\unsubst{x}{t}, \forall x\, \dual{A} +^\alpha \dual{A}\unsubst{x}{\alpha} \},\\
& \{ \exists x\, B +^\alpha B\unsubst{x}{\alpha} , \forall x\, \dual{B} +^\beta \dual{B}\unsubst{x}{\beta} \},\\
& \exists x \exists y\, C +^\alpha ( \exists y\, C\unsubst{x}{\alpha} +^\beta C\sop\sel{x}{\alpha},\sel{y}{\beta}\scl ).
\end{align*}
which is the translation of~\cite[Figure 12]{Heijltjes10Classical} into an expansion
proof with cut. Then it can be verified by a quick calculation that the choice
of reducing either the cut on $A$ or that on $B$ first determines which of
two normal forms is obtained.

However cut-elimination in classical logic can be shown confluent on the level
of the (cut-free) expansion tree on a certain class of proofs~\cite{Hetzl12Herbrand}.
For future work we hope to use such techniques for describing a confluent reduction
in expansion proofs whose normal form is unique and most general in the sense
that it contains all other normal forms as sub-expansions.

\section{Conclusion}

We showed in this paper that a relatively simple syntactic approach to expansion proofs with cut is possible. We strived for keeping the definitions and the technical details as elementary as possible. Our effort should have set the ground for addressing open combinatorial problems such as strong normalization. The price to pay for simpler reductions, however, is that we duplicate more than in Heijltjes' proof forests. This issue could be solved by an operation of merging copies of similar trees, but that should rather be understood as an optimization, rather than a theoretical necessity. Moreover, merging tends to destroy the connection with operational game semantics.  In either case, however,  we do not see a perfect correspondence between our cut-elimination process and Coquand style plays.  Copying part of the old expansion proof during the quantifier reduction step, in particular, does not seem to admit a game theoretic reading. Heijltjes' proof forests, on the contrary, avoid this copy. However, a perfect correspondence with game semantics is still a general open problem, as neither proof forests nor Herbrand nets enjoy one.

%In this paper we have presented expansion proofs with cut
%for full first-order logic including non-prenex formulas. Our definitions
%extend the existing notion of cut-free expansion proofs in a natural
%way, in particular, we kept Miller's correctness criterion for expansion trees unchanged.
%We have proved weak normalization, strong normalization remains an open problem.

%We believe that expansion proofs with cut are a useful tool for speaking about proofs in first-order logic,
%since they are compact, focus on the first-order level of proofs, and admit natural cut-reduction rules
%which are weakly normalizing --- and perhaps even strongly normalizing.
\vskip 0.2in
{\bf Acknowledgements.}
The authors would like to thank M.~Baaz, K.~Chaudhuri, W.~Heijltjes, R.~McKinley,
D.~Miller, and H. Moeneclaey for many helpful discussions about expansion trees, the
$\varepsilon$-calculus, proof forests and Herbrand nets.

\bibliographystyle{plain}
\bibliography{references}

\end{document}